\documentclass[12pt]{amsart}

\usepackage{amssymb}

\usepackage{enumitem}

\usepackage{amsmath}
\usepackage{latexsym}
\usepackage{extarrows}
\usepackage{enumerate}
\usepackage{txfonts}
\usepackage{mathtools}
\usepackage{bm}
\usepackage{tikz-cd}
\usepackage{xcolor}
\usepackage{quiver}

\usepackage[T1]{fontenc}

\newtheorem{theorem}{Theorem}[section]

\newtheorem{lemma}[theorem]{Lemma}
\newtheorem{proposition}[theorem]{Proposition}

\theoremstyle{definition}
\newtheorem{definition}[theorem]{Definition}
\newtheorem{remark}[theorem]{Remark}

\newcommand\E{\mathbb{E}}
\newcommand\Z{\mathbb{Z}}

\newcommand\C{\mathbb{C}}

\newcommand\N{\mathbb{N}}

\newcommand\op{{\operatorname{op}}}
\newcommand\Cond{\mathtt{Cond}}
\newcommand\CondTilde{{\widetilde{\mathtt{Cond}}}}
\newcommand\Aut{\operatorname{Aut}}

\newcommand\eps{\varepsilon}
\newcommand\Baire{\mathcal{B}a}
\newcommand\Borel{\mathcal{B}o}
\newcommand\AP{\mathtt{AP}_{X|Y}}
\newcommand\WM{\mathtt{WM}_{X|Y}}

\newcommand\Cat{\mathcal{C}}
\newcommand\AbsMes{\mathbf{AbsMbl}}
\newcommand\CH{\mathbf{CH}}
\newcommand\CHProb{{\mathbf{CHPrb}}}
\newcommand\CHProbG{{\mathbf{CHPrb}_\Gamma}}
\newcommand\SigmaAlg{{\mathbf{Bool}_\sigma}}
\newcommand\ProbAlg{{\mathbf{PrbAlg}}}
\newcommand\OpProbAlg{{\mathbf{PrbAlg}}}
\newcommand\OpProbAlgG{{\mathbf{PrbAlg}_\Gamma}}
\newcommand\CStarAlg{{\mathbf{CC^*Alg}}}
\newcommand\CStarAlgUnit{{\mathbf{CC^*Alg}}}
\newcommand\CStarAlgTrace{{\mathbf{CC^*Alg}^\tau}}
\newcommand\CStarAlgUnitTrace{{\mathbf{CC^*Alg}^\tau_{\Gamma^\op}}}
\newcommand\vonNeumann{{\mathbf{CvNAlg}^\tau}}
\newcommand\vonNeumannG{{\mathbf{CvNAlg}^\tau_{\Gamma^\op}}}

\newcommand\Spec{\mathtt{Spec}}
\newcommand\Stone{\mathtt{Conc}}
\newcommand\Idem{\mathtt{Proj}}
\newcommand\Riesz{\mathtt{Riesz}}
\newcommand\CFunc{C}
\newcommand\Linfty{L^\infty}
\newcommand\Inv{\mathtt{Inv}_\Gamma}

\newcommand\cip[2]{\langle #1,#2\rangle_{L^2(X|Y)}}
\newcommand\cnorm[1]{\|#1\|_{L^2(X|Y)}}
\newcommand\cltwo{\mathtt{L^2(X|Y)}}
\newcommand\chs{\mathtt{HS(X|Y)}}
\newcommand\orb{\mathtt{Orb}_\Gamma}
\newcommand\hsnorm[1]{\|#1\|_{\mathtt{HS}(X|Y)}}

\newcommand\pmet{\mathtt{d}_{\cltwo}}
\newcommand\pmeths{\mathtt{d}_{\chs}}
\newcommand\cdim{\mathtt{cdim}}

\begin{document}


\baselineskip=17pt


\title{An uncountable Furstenberg--Zimmer structure theory}

\author[A. Jamneshan]{Asgar Jamneshan}
\address{University of Bonn, Mathematical Institute, Endenicher Allee 60, D-53115 Bonn}
\email{ajamnesh@math.uni-bonn.de}


\begin{abstract} 
Furstenberg--Zimmer structure theory refers to the extension of the dichotomy between the compact and weakly mixing parts of a measure preserving dynamical system and the algebraic and geometric descriptions of such parts to a conditional setting, where such dichotomy is established relative to a factor and conditional analogues of those algebraic and geometric descriptions are sought. 
Although the unconditional dichotomy and the characterizations are known for arbitrary systems, the relative situation is understood under certain countability and separability hypotheses on the underlying groups and spaces.  
The aim of this article is to remove these restrictions in the relative situation and establish a Furstenberg--Zimmer structure theory in full generality. 
As an independent byproduct, we establish a connection between the relative analysis of systems in ergodic theory and the internal logic in certain Boolean topoi. 
\end{abstract}



\maketitle

\setcounter{tocdepth}{1}

\section{Introduction} \label{sec-intro}

\subsection{Motivation}
In the pioneering work \cite{furstenberg1977ergodic}, Furstenberg applied the countable\footnote{Here countable refers to systems of countable measure-theoretic complexity such as measure-preserving $\mathbb{Z}$-actions on standard Borel probability spaces. The precise meaning of countable measure-theoretic complexity is given later in the introduction.} Furstenberg--Zimmer structure theory, developed by him in the same article and by Zimmer in \cite{zimmer1976ergodic,zimmer1976extension} around the same time, to prove multiple recurrence for $\Z$-actions and show that this multiple recurrence theorem is equivalent to  Szemer\'edi's theorem \cite{szemeredi1975sets}. 
The Furstenberg--Zimmer structure theory is significantly refined in the Host--Kra--Ziegler structure theory \cite{host2005nonconventional,ziegler2007universal} which establishes a  description of the characteristic factors governing multiple recurrence for $\Z$-actions as inverse limits of rotations on nilmanifolds. 
The finitary counterpart of the Host--Kra--Ziegler structure theorem is known as the inverse theorem for the Gowers uniformity norms by Green, Tao, and Ziegler \cite{gtz} - a foundational result in the area of higher-order Fourier analysis initiated by Gower's seminal work \cite{gowers1,gowers2} on Szemer\'edi's theorem. 

Beyond the natural (intrinsic) motivation of generalizing ergodic theory to uncountable systems, we hope that an uncountable Host--Kra--Ziegler structure theory (such as for the action of hyperfinite abelian groups on Loeb probability spaces) may help to clarify the relationship between the ergodic-theoretical and analytic (higher-order Fourier analysis) approaches to Szemer\'edi's theorem. See \cite{jt21-1,jst} for some recent progress. 

By proving an uncountable Moore--Schmidt theorem \cite{jt19} and an uncountable Mackey-Zimmer theorem \cite{jt20} (and clarifying in \cite{jt-foundational} some foundational aspects of uncountable measure theory), first basic results in ergodic structure theory were successfully extended to uncountable settings.  The present paper aims to establish another part  by developing an uncountable Furstenberg--Zimmer structure theory. 

The uncountable Furstenberg--Zimmer structure theory of this paper was applied in the joint work \cite{roth} by the author and others to obtain an uncountable and uniform extension of the double recurrence theorem for amenable groups by Bergelson, McCutcheon, and Zhang \cite{bergelson1997roth}.

\subsection{Main results}

Let $\Gamma$ be a group and $X$ be a $\sigma$-complete Boolean algebra equipped with a countably additive measure $\mu$ of total mass $1$. 
Suppose that $\Gamma$ acts on the probability algebra\footnote{In the ergodic theory literature, it is common to call the pair $(X,\mu)$, where $X$ is a $\sigma$-complete Boolean algebra and $\mu$ is a countably additive probability measure on $X$, a \emph{measure algebra}, e.g., see \cite{glasner2015ergodic,kerrli}. We follow the conventions in Fremlin \cite{fremlinvol3} to call $(X,\mu)$ a \emph{probability algebra} and reserve the name measure algebra for abstract measures of arbitrary total mass.} $(X,\mu)$ by a group homomorphism\footnote{This definition is equivalent to the on given in Definition \ref{def-opprobalg} where by reversing the arrows in the category of probability algebras we can work with an action of $\Gamma$ instead of its opposite.} $T:\Gamma^\op\to \Aut(X,\mu)$, where $\Aut(X,\mu)$ denotes the automorphism group of $(X,\mu)$ and $\Gamma^\op$ the opposite group. The automorphisms of $(X,\mu)$ are measure-preserving bijective Boolean homomorphisms of $X$.  
We call $\mathcal{X}=(X,\mu,T)$ a \emph{probability algebra $\Gamma$-dynamical system}. 
The category of such dynamical systems and some basic tools to study them are collected in the preliminaries in Section \ref{sec-prelim}. 
A main goal of this paper is to establish the following

\begin{theorem}[Uncountable Furstenberg--Zimmer structure theorem]\label{thm-FZ}
Let $\Gamma$ be a group. 
For any  probability algebra $\Gamma$-dynamical system $\mathcal{X}=(X,\mu,T)$ there exist an ordinal $\beta$ and for each ordinal $\alpha\leq \beta$ a factor $\mathcal{X}\to \mathcal{Y}_\alpha=(Y_\alpha,\nu_\alpha,S_\alpha)$ satisfying the following properties: 
\begin{itemize}
\item[(i)] $\mathcal{Y}_0$ is trivial.  
\item[(ii)] $\mathcal{Y}_{\alpha+1}\to \mathcal{Y}_{\alpha}$ is a relatively compact extension for every successor ordinal $\alpha+1\leq \beta$. 
\item[(iii)] $\mathcal{Y}_\alpha$ is the inverse limit of the systems $\mathcal{Y}_{\eta}$, $\eta<\alpha$ for every limit ordinal $\alpha\leq \beta$, in the sense that $Y_\alpha$ is generated by $\bigcup_{\eta<\alpha} Y_\eta$ as a $\sigma$-complete Boolean algebra.     
\item[(iv)]  $\mathcal{X}\to \mathcal{Y}_\beta$ is a weakly mixing extension.  
\end{itemize}
\end{theorem}

Theorem \ref{thm-FZ} follows  from the following relative dichotomy between relatively weakly mixing and relatively compact extensions of  probability algebra $\Gamma$-dynamical systems by  transfinite induction. 

\begin{theorem}[Uncountable relative dichotomy]\label{thm-dichotomy1}
Let $\mathcal{X}=(X,\mu,T)$ and $\mathcal{Y}=(Y,\nu,S)$ be  probability algebra $\Gamma$-dynamical systems and $\pi: \mathcal{X}\to \mathcal{Y}$ an extension.  
Exactly one of the following statements is true.  
\begin{itemize}
\item[(i)] $\pi: \mathcal{X}\to \mathcal{Y}$ is a relatively weakly mixing extension.  
\item[(ii)] There exist an  probability algebra $\Gamma$-dynamical system $\mathcal{Z}=(Z,\lambda,R)$ and extensions $\phi:\mathcal{X}\to \mathcal{Z}$ and $\psi:\mathcal{Z}\to \mathcal{Y}$ such that $\psi$ is a non-trivial relatively compact extension.  
\end{itemize}
\end{theorem}

We prove Theorem \ref{thm-dichotomy1} and Theorem \ref{thm-FZ} in Section \ref{sec-structure}. Compact and weakly mixing extensions of  probability algebra systems are defined in Sections \ref{sec-compact} and \ref{sec-structure} respectively.  
Moreover, we establish various characterizations of relatively compact extensions in terms of (a) certain invariant finitely generated modules, (b) non-trivial invariant conditional Hilbert-Schmidt operators, and (c) conditionally almost periodic orbits. 
We establish the equivalence of (a)-(c) in smaller conditional $L^\infty$-modules similarly to the classical theory and also in larger conditional $L^0$-modules (where $L^0$ denotes the ring of equivalence classes of all complex measurable functions), and we show that these descriptions in the larger and the smaller modules are equivalent.  
In the case of ergodic systems, we prove  in Section \ref{sec-isometric} that relatively compact extensions are isomorphic to isometric extensions (that is,  probability algebra homogeneous skew-product systems as introduced in \cite{jt20}).   

\begin{table}
\begin{tabular}{|c|c|}
\hline
\textbf{Countable complexity} & \textbf{Uncountable complexity} \\
\hline
\hline
Countable group actions & Uncountable group actions \\
\hline
Standard Borel probability spaces & Inseparable probability algebras  \\ 
\hline 
Pointwise action by measurable &  Pointfree action by  Boolean \\
 functions & homomorphisms \\ 
\hline
Cantor models & Canonical Stonean models \\
\hline
Borel measurability & Baire measurability \\
\hline
Homogeneous skew-product extensions  & Homogeneous skew-product extensions  \\
 by metrizable compact groups  & by compact Hausdorff groups \\
\hline 
von Neumann mean ergodic theorem & Alaoglu-Birkhoff ergodic theorem \\
\hline 
 Countable Furstenberg towers & Uncountable Furstenberg towers \\ 
\hline
Classical disintegration of measures & Canonical disintegration of measures \\
\hline
measurable Hilbert bundles & Conditional Hilbert spaces \\ 
\hline
Measurable selection theory & Conditional analysis \\ 
\hline
\hline
\end{tabular}
\caption{\label{tab-complexity} The group-theoretic, measure-theoretic, topological and methodological differences of systems of countable and uncountable complexity (with a view towards their  Furstenberg--Zimmer structure theory).} 
\end{table}

These results are well understood for systems of countable measure-theoretic complexity, see  \cite{furstenberg1977ergodic,zimmer1976ergodic,zimmer1976extension} for the original papers and \cite{einsiedler2010ergodic,furstenberg2014recurrence,glasner2015ergodic,kerrli,tao2009poincare} for some textbook expositions.  The aim of this paper is to extend them to systems of uncountable measure-theoretic complexity. We say that a probability algebra $\Gamma$-dynamical system $(X,\mu,T)$ has \emph{countable measure-theoretic complexity} if $\Gamma$ is countable and $(X,\mu)$ is separable\footnote{We have a metric structure on $X$ defined by $d(E,F)=\mu(E\Delta F)$.}, and has \emph{uncountable measure-theoretic complexity} otherwise. In Table \ref{tab-complexity}, we compare some features of systems of countable and uncountable complexity. 

We establish a Furstenberg--Zimmer structure theory for not necessarily countable groups and not necessarily separable probability algebras. 
To achieve such generality we have to introduce novel tools to study the  structure of probability algebra systems relative to some factor since many classical tools such as disintegration of measures, measurable selection lemmas, and the theory of measurable Hilbert bundles rely on such countability and separability hypotheses.  Instead, we apply tools from topos theory as suggested in \cite{taotopos} by Tao.  
We will not use topos theory directly nor define a sheaf anywhere in this paper.   
Instead, we use the closely related conditional analysis as developed in \cite{cheridito2015conditional,drapeau2016algebra,filipovic2009separation,jamneshan2018measure}.  
A main advantage of conditional analysis is that it can be setup without much costs and understood with basic knowledge in measure theory and functional analysis.   
The ergodic theoretic results of this paper are organized in a self-contained presentation and no familiarity with topos theory is required for the reader purely interested in the ergodic theoretic content. In fact, the use of conditional analysis and the greater generality lead to significant simplifications of the proofs of several results in  Furstenberg--Zimmer structure theory. 

For the readers familiar with topos theory and interested in the connection between ergodic theory and topos theory, we include several remarks relating the conditional analysis of this paper to the internal discourse in Boolean Grothendieck sheaf topoi. 
To the best knowledge of the author these remarks comprise novel insight into the semantics of such topoi.

We would like to point out that in the recent work \cite{ehk} a different proof of Theorem \ref{thm-FZ} is provided which is based on a functional analytic decomposition theorem for group representations on Kaplansky-Hilbert modules. 

\subsection{Technical innovations} 

\emph{Apriori} probability algebras are not $\sigma$-algebras of subsets of an underlying set. 
We view probability algebras as abstract pointfree probability spaces as opposed to concrete point-set probability spaces defined on $\sigma$-algebras of sets. 
In countable ergodic theory, the starting point are concrete countably generated probability spaces on which a measure-preserving action by a countable or second-countable group on the underlying set  is defined. 
In the countable theory, a distinction is made between an action which is defined everywhere on the underlying set and a near-action which is only defined almost everywhere, e.g., see \cite{zimmer1976extension}. 
This distinction is non-trivial, since there are natural examples of systems of countable complexity for Polish group near-actions that cannot be realized by a pointwise action, e.g., see \cite{gtw,mooresolecki}. 

In uncountable ergodic theory, the starting point are arbitrary (not necessarily countably generated or separable) probability algebras on which an arbitrary (not necessarily countable or second-countable) group acts by measure-preserving Boolean isomorphisms. 
A major advantage of working with pointfree probability algebras is that one  systematically avoids to have to deal with null set issues. 
In this regard, the abstract pointfree perspective, which has been already implemented in our previous works \cite{jt19,jt20} and systematically developed in the foundational paper \cite{jt-foundational}, seems to be more natural in uncountable ergodic theory. 

Furthermore, the category of probability algebra dynamical systems is dually equivalent to the dynamical category of commutative tracial von Neumann algebras (see Figure \ref{fig:canon}) such that all structural results established in this paper for probability algebra dynamical systems have natural analogues in the dynamical category of commutative tracial von Neumann algebras, see \cite[\S 3.2]{kerrli} for this correspondence for systems of countable complexity.  
See \cite[\S 2]{popa} for a proof of a non-commutative Furstenberg--Zimmer type dichotomy of relatively weakly mixing and relatively compact extensions for countable group actions on separable von Neumann algebras. 
In fact, the analogies extend also to the other two dynamical categories of commutative tracial $C^*$-algebras and topological dynamical systems preserving a Baire-Radon probability measure in Figure \ref{fig:canon} by restricting to the image of the $\Stone$ and $\mathtt{Inc}$ functors, see Section \ref{sec-prelim} for the definition of these functors. 

On the other hand, by quotienting out the null ideal, to any concrete countable complexity system we can associate an abstract probability algebra system. 
In this regard, the abstract perspective is implicitly present in the countable theory, e.g., see \cite[\S 5]{furstenberg2014recurrence}, \cite[\S 2]{glasner2015ergodic}, the preliminaries of \cite{zimmer1976ergodic,zimmer1976extension}, and the expositions in \cite{tao2009poincare,kerrli}.

In order to study the Furstenberg--Zimmer structure theory of abstract pointfree systems analogues of some basic and more advanced tools from measure theory and functional analysis are required. 
These are basic Hilbert theory of $L^2$-spaces, disintegration of measures, and the theory of measurable Hilbert bundles.  
However, classical disintegration of measures and the theory of measurable Hilbert bundles are in general only applicable for systems of countable measure-theoretic complexity. 
Therefore, the main technical difficulty in extending the Furstenberg--Zimmer structure theory is to find viable alternatives which work for abstract systems of uncountable measure-theoretic complexity as well. 
We divide the above tools into two sets. The first set are basic tools from measure theory and the second one consists of tools needed for the analysis relative to some factor. 

For the first group we have two options. 
We can either use the abstract measure theory for measure algebras as systematically developed in \cite{fremlinvol3}, or we can use concrete measure theory after passing to a \emph{canonical model}. 
The canonical model, which is introduced in Section \ref{sec-prelim} (see also \cite{EFHN} and the references given in Remark \ref{rem-conc}), is a functor associating to any probability algebra dynamical system a unique compact Hausdorff probability measure-preserving dynamical system where the underlying group acts by homeomorphisms. The canonical model comes  with some favorable properties, for example a canonical disintegration of measures, see Section \ref{sec-canonical}. 

We replace the second set of tools by conditional analysis which permits to perform analysis in relative situations without any countability or separability assumptions. 
More precisely, we introduce a conditional Hilbert space and its conditional tensor product representing conditional Hilbert--Schmidt operators, and we study a conditional Gram--Schmidt process and conditional spectral properties of these conditional Hilbert--Schmidt operators in Section \ref{sec-condanal}. 

We found it convenient to explicitly use the language of category theory in this paper. It helps to clearly separate the concrete from the abstract probability spaces, and clarifies the relationships between these categories by way of \emph{functorial relations}. These functors work particularly well with dynamics. This allows to safely pass back and forth between these categories, and thus take advantage of tools available in both the concrete and abstract settings. Functors establishing dualities of categories such as the useful Gelfand duality permit to efficiently translate into functional analytic categories to a similar effect. Practically, these categorical relations help with a basic housekeeping, avoiding making unnecessary mistakes. An illustrating example may be the clarification of category-dependent notions of measurability and various product constructions whose relations to each other can be worked out by functorial relations, cf.~\cite{jt-foundational, jt20}.  In fact, once these categorical relations are explicitly established and well understood, certain frequently used constructions and arguments automatize and simplify.

\subsection*{Acknowledgements}
The author was supported by DFG-research fellowship JA 2512/3-1. 
The author offers his thanks to Terence Tao for suggesting this project, many helpful discussions, and his encouragement and support.  
He is grateful to Pieter Spaas for several helpful discussions.    
The author thanks Markus Haase for organizing an online workshop on structural ergodic theory where the results of this paper and the parallel work \cite{ehk} could be discussed, and Nikolai Edeko, Markus Haase and Henrik Kreidler for helpful comments on an early version of the manuscript. The author is indebted to the anonymous referee for several useful suggestions and corrections.

\section{Preliminaries}\label{sec-prelim}

\subsection{The categories}

Throughout this paper, we fix an arbitrary group $\Gamma$.  
We define the category $\OpProbAlgG$ of probability algebra $\Gamma$-dynamical systems to be the opposite of a category of Boolean algebras equipped with a probability measure-preserving dynamical structure. 

A thorough discussion of some categorical aspects of ergodic theory and adjacent areas in measure theory and functional analysis, such as the auxiliary categories depicted in Figure \ref{fig:canon}, together with useful references to basic category theory can be found in the accompanying paper \cite{jt-foundational}, to the relevant parts of which we will provide exact referencing as we proceed. The interested reader is also referred to \cite{jt-foundational} for our categorical notations and conventions. 

\begin{definition}[The category $\OpProbAlgG$]\label{def-opprobalg} \
\begin{itemize}
    \item[(i)] A $\SigmaAlg$-object is a \emph{$\sigma$-complete Boolean algebra} $X=(X,\vee,\wedge, \bar \cdot, 0,1)$ and a $\SigmaAlg$-morphism from a $\SigmaAlg$-algebra $X$ to another $\SigmaAlg$-algebra $Y$ is a \emph{$\sigma$-complete Boolean homomorphism} $f:X\to Y$. 
    \item[(ii)] The category $\AbsMes$ of \emph{abstract measurable spaces} is the opposite category of $\SigmaAlg$. 
    \item[(iii)] The objects of $\ProbAlg$ are \emph{probability algebras} $(X,\mu)$ where $X$ is a $\SigmaAlg$-algebra and $\mu:X\to [0,1]$ is a countably additive probability measure, that is (i) $\mu(\bigvee_n E_n)=\sum_n \mu(E_n)$ for any sequence $(E_n)$ of pairwise disjoint elements in $X$, (ii) $\mu(X)=1$, and (iii) $\mu(E)=0$ if and only if $E=0$. A \emph{$\ProbAlg$-morphism} from a $\ProbAlg$-algebra $(X,\mu)$ to another $\ProbAlg$-algebra $(Y,\nu)$ is a Boolean algebra homomorphism\footnote{Notice the opposite direction of the arrow. We implicitly work with an opposite category here in order to keep certain functors covariant.} $\pi:Y\to X$ such that $\mu_*\pi(E)=\nu(E)$ for all $E\in Y$, where we denote by $\mu_*\pi$ the pullback probability measure on $Y$ defined by $\mu_*\pi(E)\coloneqq \mu(\pi(E))$. The $\SigmaAlg$-algebra $X$ of a  $\ProbAlg$-object automatically upgrades to a complete Boolean algebra (that is, a Boolean algebra in which arbitrary joins and meets exists) and the underlying Boolean algebra homomorphism of a $\ProbAlg$-morphism automatically upgrades to a complete Boolean algebra homomorphism (that is, one that preserves arbitrary joins and meets), see Remark \ref{rem-ccc}. Let $\Aut(X,\mu)$ denote the automorphism group of a $\ProbAlg$-algebra $(X,\mu)$ in the category $\ProbAlg$. 
    \item[(iv)] A \emph{$\OpProbAlgG$-system} is a tuple $(X,\mu,T)$ where $(X,\mu)$ is a $\ProbAlg$-algebra and $T:\Gamma \to \Aut(X,\mu)$ is a group homomorphism $\gamma \mapsto T_{\gamma}$. A \emph{$\OpProbAlgG$-morphism} from a $\OpProbAlgG$-system $(X,\mu,T)$ to another $\OpProbAlgG$-system $(Y,\nu,S)$ is a $\OpProbAlg$-morphism $\pi:(X,\mu)\to (Y,\nu)$ such that $\pi\circ T_{\gamma}=S_{\gamma}\circ \pi$ for all $\gamma\in\Gamma$. We call $\OpProbAlgG$  the category of \emph{probability algebra $\Gamma$-dynamical systems}.  
We denote a $\OpProbAlgG$-system by $\mathcal{X}=(X,\mu,T)$.  
If $\pi:\mathcal{X}\to \mathcal{Y}$ is a $\OpProbAlgG$-morphism, we call $\mathcal{X}$ a \emph{$\OpProbAlgG$-extension} of $\mathcal{Y}$ and $\mathcal{Y}$ a \emph{$\OpProbAlgG$-factor} of $\mathcal{X}$. 
We sometimes also refer to the $\OpProbAlgG$-morphism $\pi$ itself as a $\OpProbAlgG$-extension or a $\OpProbAlgG$-factor. 
\end{itemize}
 An extensive discussion of the previously introduced categories can be found in \cite[\S 6]{jt-foundational}. 
\end{definition}

\begin{remark}\label{rem-ccc}
It is well known that the $\SigmaAlg$-algebra $X$ of a $\ProbAlg$-algebra  $(X,\mu)$ is automatically a complete Boolean algebra. 
We provide the short proof for the convenience of the reader. 
Let $(E_i)_{i\in I}$ be an arbitrary family in $X$. 
Well-order $I$ by $<$ and let $F_i=E_i\wedge\overline{\bigvee_{j<i} E_j}$ for each $i\in I$. 
By construction, $(F_i)_{i\in I}$ is a family of pairwise disjoint elements of $X$. 
Since $\mu$ is a finite measure, only countably many of the $F_i$ are nonzero (this property is also called the \emph{countable chain condition}). 
Thus $\bigvee_{i\in I} E_i=\bigvee_{i\in I, F_i> 0} F_i$, and by $\sigma$-completeness of $X$, $\bigvee_{i\in I, F_i\neq 0} F_i\in X$. 
Moreover if $\pi:(Y,\nu)\to (X,\mu)$ is a $\ProbAlg$-morphism (the underlying $\SigmaAlg$-morphism of which does not necessarily preserve countable joins and meets by definition), then $\pi$ automatically upgrades  to a complete Boolean homomorphism due to the automatic injectivity of $\pi$ and the countable chain condition, cf. \cite[Remark 1.2]{jt20}.  
\end{remark}

Next we introduce some basic tools to study the structure of $\OpProbAlgG$-systems. 
We obtain these tools by relating the category $\OpProbAlgG$ to some adjacent auxiliary dynamical categories of algebras and spaces as illustrated in Figure \ref{fig:canon}. The symbol $\op$ on a category indicates the use of opposite category, while $\Gamma^\op$ for a group $\Gamma$ means the use the opposite group. 

  \begin{figure}
    \centering
    \begin{tikzcd}
 (\vonNeumannG)^\op \arrow[r,tail, two heads, "\mathtt{Inc}"] \arrow[d, tail, "\Idem", two heads, shift left=0.75ex]
& 
(\CStarAlgUnitTrace)^\op \arrow[d,tail,two heads,"\Riesz",shift left=.75ex]  \\
 \OpProbAlgG  \arrow[r,"\Stone"', tail, two heads] \arrow[u,"\Linfty",tail, two heads, shift left=0.75ex] & \CHProbG \arrow[u,"\CFunc", two heads,  tail, shift left=.75ex] 
\end{tikzcd}
    \caption{The adjacent auxiliary dynamical categories. Tailed arrows are faithful functors, while arrows with two heads are full functors. The diagram commutes up to natural isomorphisms.}
    \label{fig:canon}
\end{figure}

\begin{definition}[Adjacent dynamical categories]
First we introduce the remaining categories in Figure \ref{fig:canon} without a dynamical structure. 
Then we apply a universal procedure to turn these categories into corresponding dynamical categories.  
\begin{itemize}
    \item[(i)] A $\CHProb$-space is a tuple $(X,\Baire(X),\mu)$ where $X$ is a compact Hausdorff space, $\Baire(X)$ is the Baire $\sigma$-algebra of $X$ (that is, the smallest $\sigma$-algebra on $X$ rendering continuous functions measurable), and $\mu$ is a Baire-Radon probability measure on the measurable space $(X,\Baire(X))$. 
    A $\CHProb$-morphism $\pi:(X,\Baire(X),\mu)\to (Y,\Baire(Y),\nu)$ between compact Hausdorff probability spaces is a continuous function $\pi:X\to Y$ such that $\pi_\ast\mu=\nu$, where $\pi_\ast\mu$ denotes pushforward measure (a continuous function between compact Hausdorff spaces is always Baire-measurable since its  pullback preservers compact $G_\delta$ sets).    
    \item[(ii)] A \emph{$\CStarAlgTrace$-algebra} $(\mathcal{A},\tau)$ is a unital commutative $C^*$-algebra equipped with a state $\tau:\mathcal{A}\to \C$, that is a bounded linear functional which is non-negative (it maps non-negative elements to non-negative reals) and is of operator norm $1$.   
   A \emph{$\CStarAlgTrace$-morphism} $\Phi:(\mathcal{A},\tau_\mathcal{A})\to (\mathcal{B},\tau_{\mathcal{B}})$ between $\CStarAlgTrace$-algebras is a unital $*$-homomorphism $\Phi:\mathcal{A}\to \mathcal{B}$ such that $\tau_\mathcal{B}\circ \Phi=\tau_\mathcal{A}$. 
   \item[(iii)] A $\vonNeumann$-algebra $(\mathcal{A},\tau_\mathcal{A})$ is a commutative von Neumann algebra $\mathcal{A}$ equipped with a faithful trace $\tau_\mathcal{A}$, that is to say a $*$-linear functional $\tau_\mathcal{A}: \mathcal{A}\to \C$ with $\tau_\mathcal{A}(1) = 1$, and $\tau_\mathcal{A}(aa^*)\geq 0$
for any $a \in\mathcal{A}$,  with equality if and only if $a = 0$. A $\vonNeumann$-morphism $\Phi:(\mathcal{A},\tau_\mathcal{A})\to (\mathcal{B},\tau_\mathcal{B})$ between $\vonNeumann$-algebras is a von Neumann algebra
homomorphism $\Phi:\mathcal{A}\to \mathcal{B}$ such that $\tau_\mathcal{A}=\tau_\mathcal{B}\circ \Phi$. 
\end{itemize}
Let $\Cat$ be one of the categories $\CHProb,\CStarAlgTrace$, or $\vonNeumann$. 
For a $\Cat$-object $X$, let $\Aut_\Cat(X)$ be the automorphism group of $X$ in $\Cat$. 
We define a $\Cat_\Gamma$-object (resp. a $\Cat_{\Gamma^\op}$-object) to be a tuple $(X,T)$, where $X$ is a $\Cat$-object and $T:\Gamma\to \Aut_\Cat(X)$ is a homomorphism $\gamma\mapsto T^\gamma$ of groups (resp. $T:\Gamma^\op\to \Aut_\Cat(X)$ is a homomorphism $\gamma^\op\mapsto T^{\gamma^\op}$ of groups). A $\Cat_\Gamma$-morphism $\pi:(X,T)\to (Y,S)$ (resp. $\Cat_{\Gamma^\op}$-morphism $\pi:(Y,S)\to (X,T)$) between $\Cat_\Gamma$-objects (resp. $\Cat_{\Gamma^\op}$-objects) is a $\Cat$-morphism $\pi:X\to Y$ (resp. a $\Cat$-morphism $\pi:Y\to X$) obeying the intertwining property $\pi\circ T^\gamma=S^\gamma\circ \pi$ for all $\gamma\in\Gamma$ (resp. $\pi\circ S^{\gamma^\op}=T^{\gamma^\op}\circ \pi$ for all $\gamma^\op\in\Gamma^\op$).  
\end{definition}

\subsection{The functors}

The functors in Figure \ref{fig:canon}, in particular $\Linfty$, $\CFunc$, $\Idem$, $\Riesz$, and $\Stone$, are the means by which  we represent the objects and morphisms in the dynamical category $\OpProbAlgG$ by objects and morphisms in the adjacent dynamical categories $\CHProbG,\vonNeumannG,\CStarAlgUnitTrace$ in order to use the tools available in the latter categories to study structural properties of systems in the former, while the canonical model functor $\Stone$ is our main player. We start by describing the two Gelfand-type dualities in Figure \ref{fig:duality} which yield the dualities of the dynamical categories in Figure \ref{fig:canon} from which we then derive the canonical model functor $\Stone$ in the same figure. 

  \begin{figure}
    \centering
    \begin{tikzcd}
 (\vonNeumann)^\op \arrow[d,  tail, "\Idem", two heads, shift left=0.75ex]
& (\CStarAlgTrace)^\op \arrow[d,tail,two heads,"\Riesz",shift left=.75ex]  \\
 \OpProbAlg   \arrow[u,"\Linfty",tail, two heads, shift left=0.75ex] & \CHProb \arrow[u,"\CFunc", two heads, tail, shift left=.75ex]
\end{tikzcd}
  \caption{The Gelfand-type dualities between probability algebras and tracial commutative von Neumann algebra and between compact Hausdorff probability spaces and tracial commutative unital $C^*$-algebras.}
    \label{fig:duality}
\end{figure}
Gelfand duality (cf. Figure \ref{fig:gelfand}) establishes a well known equivalence between the category $\CH$ of compact Hausdorff spaces and the category $\CStarAlg$ of unital commutative $C^*$-algebras. The functor $\Spec$ associates to each unital commutative $C^*$-algebra $\mathcal{A}$ its spectrum $\Spec(\mathcal{A})$ (the set of $\CStarAlg$-morphisms from $\mathcal{A}$ to $\C$) which is a compact Hausdorff space by the Banach-Alaoglu theorem and to each unital $*$-homomorphism $\Phi:\mathcal{A}\to \mathcal{B}$ the continuous function $\Spec(\Phi):\Spec(\mathcal{B})\to \Spec(\mathcal{A})$, $\Spec(\Phi)(f):=f\circ \Phi$. 
The functor $\CFunc$ associates with each compact Hausdorff space $X$ the unital commutative $C^*$-algebra $C(X)$ of continuous functions $f:X\to \C$ and with each continuous function $\pi:X\to Y$ between $\CH$-spaces the Koopmann operator $C(\pi):C(Y)\to C(X)$, $C(\pi)(f):=f\circ \pi$. See \cite[\S 2]{jt-foundational} for a more detailed exposition.  

The Gelfand duality in Figure \ref{fig:gelfand} extends to a "Riesz duality" between $\CHProb$ and $\CStarAlgTrace$ as follows. 
A $\CHProb$-space $(X,\Baire(X),\mu)$ is mapped via the $\CFunc$-functor to the $\CStarAlgTrace$-algebra $(C(X),\tau_\mu)$ where the state $\tau_\mu:C(X)\to \C$ is defined by integration $\tau_\mu(f):=\int_X fd\mu$. If $\pi:(X,\Baire(X),\mu)\to (Y,\Baire(Y),\nu)$ is a $\CHProb$-morphism, then the Koopman operator $C(\pi):C(Y)\to C(X)$ satisfies $\tau_\nu=\tau_\mu\circ C(\pi)$.
Conversely, using Gelfand duality and the Riesz representation theorem we can associate to each $\CStarAlgTrace$-algebra $(\mathcal{A},\tau_\mathcal{A})$ a unique $\CHProb$-space $\Riesz(\mathcal{A},\tau_\mathcal{A})\coloneqq(\Spec(\mathcal{A}),\Baire(\Spec(\mathcal{A})),\mu_{\tau_\mathcal{A}})$. Uniqueness in the Riesz representation theorem also yields that if $\Phi:(\mathcal{A},\tau_\mathcal{A})\to (\mathcal{B},\tau_\mathcal{B})$ is a $\CStarAlgTrace$-morphism, then the continuous function $\Spec(\Phi):\Spec(\mathcal{B})\to \Spec(\mathcal{A})$ pushes forward $\nu_{\tau_\mathcal{B}}$ to $\mu_{\tau_\mathcal{A}}$. 
This "Riesz duality" between $\CHProb$ and $\CStarAlgTrace$ can be promoted in the obvious way to a duality between the corresponding dynamical categories $\CHProbG$ and $\CStarAlgUnitTrace$. This establishes the duality on the right-hand side of Figure \ref{fig:duality}. See \cite[\S 5]{jt-foundational} for a more detailed exposition.

  \begin{figure}
    \centering
    \begin{tikzcd}
(\CStarAlg)^\op \arrow[d,tail,two heads,"\Spec",shift left=.75ex]  \\
\CH \arrow[u,"\CFunc", two heads, tail, shift left=.75ex]
\end{tikzcd}
  \caption{Gelfand duality between the category of compact Hausdorff spaces and commutative unital $C^*$-algebras.}
    \label{fig:gelfand}
\end{figure}

As for the duality on the left-hand side of Figure \ref{fig:duality}, let $(X,\mu)$ be a $\OpProbAlg$-space. 
We define $\Linfty(X)$ to be the set of $\AbsMes$-morphisms $f:X\to \Borel(\C)$, where $\Borel(\C)$ denotes the Borel $\sigma$-algebra of the complex numbers $\C$, which are bounded in the sense that there exists a real number $M\geq 0$ such that  $f( \{ z \in \C: |z| \leq M \} ) = 1$, with $\|f\|_{\Linfty(X)}$ defined to equal the infimum of all such $M$. We can equip $\Linfty(X)$ with the structure of a commutative $*$-algebra by lifting all $*$-algebra operations from $\C$ to the set of all $\AbsMes$-morphisms $f:X\to \Borel(\C)$. Every element $E$ of $X$ generates an idempotent element $1_E$ of $\Linfty(X)$. The set of finite linear combinations of idempotents form a dense subspace of $\Linfty(X)$.  One can then define a trace $\tau$ on this algebra by defining
$\tau ( \sum_{n=1}^N c_n 1_{E_n} ) \coloneqq \sum_{n=1}^N c_n \mu(E_n)$
for any finite sequence of complex numbers $c_n$ and $E_n \in X$, and then extending by density.  
Thus $\Linfty(X)$ can be viewed as an element of $\vonNeumann$.  If $\pi \colon X \to Y$ is a $\OpProbAlg$-morphism, one can define the $\vonNeumann$-morphism $\Linfty(\pi) \colon \Linfty(Y) \to \Linfty(X)$ by the Koopman operator
$\Linfty(\pi)(f) \coloneqq f \circ \pi$.  See \cite[\S 6, 7]{jt-foundational} for details of these constructions. 

Conversely, suppose that $({\mathcal A},\tau_{\mathcal A})$ is a $\vonNeumann$-algebra.  We can form the collection ${\mathcal P}_{\mathcal A}$ of real projections in ${\mathcal A}$.  As is well-known, these projections have the structure of a complete Boolean algebra.  The trace $\tau_{\mathcal A}$ then becomes a countably additive probability measure on ${\mathcal P}_{\mathcal A}$, and we write $\Idem({\mathcal A},\tau_{\mathcal A})$ for the probability algebra of $( {\mathcal P}_{\mathcal A}, \tau_{\mathcal A})$.
If $\Phi \colon ({\mathcal A},\tau_{\mathcal A}) \to ({\mathcal B},\tau_{\mathcal B})$ is a $\vonNeumann$-morphism, we observe that the associated von Neumann algebra homomorphism $\Phi \colon {\mathcal A} \to {\mathcal B}$ maps projections in ${\mathcal P}_{\mathcal A}$ to projections in ${\mathcal P}_{\mathcal B}$, in a manner that preserves the trace as well as being a $\SigmaAlg$-morphism.  We then define $\Idem(\Phi) \colon \Idem({\mathcal B},\tau_{\mathcal B}) \to \Idem({\mathcal A},\tau_{\mathcal A})$ to be the $\OpProbAlg$-morphism associated with this $\SigmaAlg$-morphism.  
See \cite[\S 7]{jt-foundational} for a proof that $\Idem$ and $\Linfty$ establish a duality of categories between $\vonNeumann$ and $\OpProbAlg$.  This duality then extends naturally to the dynamical categories $\vonNeumannG$ and $\OpProbAlgG$.  

Every von Neumann algebra is also a unital $C^*$-algebra, and a faithful trace on a commutative von Neumann algebra becomes a state on the corresponding $C^*$-algebra. From this it is easy to see that there is a forgetful inclusion functor $\mathtt{Inc}$ from $\vonNeumannG$ to $\CStarAlgUnitTrace$.

Finally, we define the \emph{canonical model functor} by
\[
\Stone:=\Riesz\circ \mathtt{Inc}\circ \Linfty.  
\]

\begin{remark}\label{rem-stone}
There is a more direct way to define the canonical model functor using the Stone and Loomis-Sikorski representation theorems where the underlying compact Hausdorff space is the Stonean space associated to a complete Boolean algebra by Stone's representation theorem, e.g., see \cite[\S 9]{jt-foundational}. 
\end{remark}

\begin{remark}\label{rem-conc}
The canonical model has been implicitly introduced in the literature at several occasions, e.g., see \cite{fremlinvol3,segal,DNP,ellis,EFHN,kakutani-l,dieudonne-counter,halmos-dieudonne} where such models are also referred to as ``Stone'' or ``Kakutani models''. 

For a countable group $\Gamma$, let us denote by $\mathbf{PrbAlg}^\sigma_\Gamma$ the subcategory of $\OpProbAlgG$ whose systems consist of separable probability algebras.  
In \cite[\S 2]{glasner2015ergodic}, a ''Cantor model'' is associated to each $\mathbf{PrbAlg}^\sigma_\Gamma$-system which satisfies some functoriality properties, (such Cantor models are formed on closed subspaces of the classical Cantor set). These Cantor models are not canonical, but amenable to classical disintegration of measures. Furstenberg \cite[\S 5-6]{furstenberg2014recurrence} works with ``separable systems'' which are concrete probability spaces on which a countable group acts measure-preservingly (modulo null sets) such that the associated $\OpProbAlgG$-systems (with the help of the $\mathtt{AlgAbs}$-functor which is described below) are in the subcategory $\mathbf{PrbAlg}^\sigma_\Gamma$. 
\end{remark}

The categories $\OpProbAlgG$ and $\CHProbG$ are not dual to each other (similarly, as the categories $\vonNeumannG$ and $\CStarAlgUnitTrace$ are not so). However we have a functor $\mathtt{AlgAbs}$ (cf. Figure \ref{fig:algabs}) that associates with every $\CHProbG$-system a $\OpProbAlgG$-system such that $\mathtt{AlgAbs}\circ \Stone$ is naturally isomorphic to the identity functor on $\OpProbAlgG$. 

  \begin{figure}
    \centering
    \begin{tikzcd}
 \OpProbAlgG   & \CHProbG \arrow[l,two heads, "\mathtt{AlgAbs}"'] 
\end{tikzcd}
    \caption{The $\mathtt{AlgAbs}$ functor first abstracts away from the set structure of a $\CHProbG$-system and then  deletes the null ideal of this abstract system to obtain a $\OpProbAlgG$-system.}
    \label{fig:algabs}
\end{figure}

Next we define the functor $\mathtt{AlgAbs}$. 
Let $(X,\Baire(X),\mu)$ be a $\CHProb$-space.
Denote by $\mathcal{N}_\mu$ the $\mu$-null ideal of $\Baire(X)$, that is the collection of $E\in\Baire(X)$ such that $\mu(E)=0$.  
We can form the quotient Boolean $\sigma$-algebra $X_\mu=\Baire(X)/\mathcal{N}_\mu$ by identifying $E,F\in\Baire(X)$ if $E\Delta F\in \mathcal{N}_\mu$. 
Let $\pi_X:\Baire(X)\to X_\mu$ be the canonical $\SigmaAlg$-epimorphism. 
We define a probability measure $\hat\mu:X_\mu\to [0,1]$ by $\hat\mu(\pi(E))\coloneqq\mu(E)$. 
We obtain a $\ProbAlg$-algebra $(X_\mu, \hat\mu)$. 
If $f:(X,\Baire(X),\mu)\to (Y,\Baire(Y),\nu)$ is a $\CHProb$-morphism, then 
the pullback map $f^*:\Baire(Y)\to \Baire(X)$ defined by $E\mapsto f^*(E)\coloneqq f^{-1}(E)$ is a $\SigmaAlg$-morphism.
We define the $\OpProbAlg$-morphism $\mathtt{AlgAbs}(f): (X_\mu, \hat\mu)\to (Y_\nu, \hat\nu)$ by $\pi_Y(E)\mapsto \mathtt{AlgAbs}(f)(\pi_Y(E))\coloneqq \pi_X(f^*(E))$.  
The functor $\mathtt{AlgAbs}$ promotes to a functor from $\CHProbG$ to $\OpProbAlgG$.  
The functor $\mathtt{AlgAbs}$ is not injective on objects since a $\CHProb$-space and its measure-theoretic completion are mapped to the same $\OpProbAlg$-space. 
Moreover,  $\Stone(\mathtt{AlgAbs}(X,\Baire(X),\mu))$ is typically much larger than $(X,\Baire(X),\mu)$.

\subsection{Abstract Lebesgue spaces}

The following notations will be used throughout the remainder of this paper. 

Let $\mathcal{X}=(X,\mu,T)$ be a $\OpProbAlgG$-system. 
We denote the canonical model $\Stone(\mathcal{X})$ by $\tilde{\mathcal{X}}$, its underlying Stonean space by $\tilde{X}$ (cf. Remark \ref{rem-stone}), its Baire-Radon measure  by $\tilde{\mu}$, and its $\Gamma$-action by $\tilde{T}$, so that $$\Stone(\mathcal{X})=\tilde{\mathcal{X}}=(\tilde{X},\Baire(\tilde{X}), \tilde{\mu},\tilde{T}).$$  
If $\pi:\mathcal{X}\to \mathcal{Y}$ is a $\OpProbAlgG$-factor, then we denote its canonical representation $\Stone(\pi)$ by $\tilde{\pi}$. 

We can use the canonical model to define integration and Lebesgue spaces for probability algebras\footnote{In \cite[\S 363-366]{fremlinvol3}, a direct construction for integration and Lebesgue spaces for measure algebras is provided which is equivalent to our construction via canonical models.}. 
For any $0\leq p\leq \infty$, we define the complex vector space $$L^p(X)=L^p(X,\mu)\coloneqq L^p(\tilde{X},\Baire(\tilde{X}), \tilde{\mu}).$$ 
For any $f\in L^1(X)$, we define $$\int_X f d\mu\coloneqq \int_{\tilde{X}} f d\tilde{\mu}.$$ 
Let $\pi:\mathcal{X}\to \mathcal{Y}$ be a $\OpProbAlgG$-factor with canonical representation $\tilde{\pi}$. 
Then $\tilde{\pi}$ gives rise to a Koopman operator $\tilde{\pi}^*:L^p(Y)\to L^p(X)$ defined by pullback as $$\tilde{\pi}^*f\coloneqq f\circ \tilde{\pi}.$$ 
In particular, we can use the Koopman operator to identify $L^p(Y)$ with a (closed) subspace of $L^p(X)$.  
Henceforth, we will make use of these identifications without further notice.

\subsection{Canonical disintegration and relatively independent products}\label{sec-canonical} 

A main advantage of the canonical model is that it yields a canonical disintegration of measures from which in turn several key basic measure-theoretic concepts can be defined. 
Canonical disintegration of measures relies on an important property which we coined the \emph{strong Lusin property} in \cite{jt-foundational} (a property well known in the literature, though under different guise). It implies that we have a $\CStarAlgUnit$-isomorphism 
\begin{equation}\label{eq-lusin}
L^\infty(X)\equiv C(\tilde{X})
\end{equation}
for any $\ProbAlg$-algebra $(X,\mu)$ with Stone space $\tilde{X}$.  
The strong Lusin property implies the following disintegration result. 
\begin{theorem}[Canonical disintegration]\label{thm-disintegration} \cite[\S 8]{jt-foundational}
Let $\mathcal{X}=(X,\mu,T)$ and $\mathcal{Y}=(Y,\nu,S)$ be $\OpProbAlgG$-systems and $\pi \colon \mathcal{X}\to \mathcal{Y}$ be a $\OpProbAlgG$-factor.  
Then there is a unique Radon probability measure $\mu_y$ on $\tilde{X}$ for each $y \in \tilde{Y}$ which depends continuously on $y$ in the vague topology in the sense that $y \mapsto \int_{\tilde{X}} f\ d\mu_y$ is continuous for every $f\in C(\tilde{X})$, and such that
\begin{equation}\label{disint-form}
\int_{\tilde X} f \tilde{\pi}^*g\ d\tilde\mu = \int_{\tilde{Y}} \left(\int_{\tilde{X}} f\ d\mu_y\right) g\ d\tilde{\nu}\end{equation}
for all $f \in C(\tilde{X})$, $g \in C(\tilde{Y})$.  
Furthermore, for each $y \in \tilde{Y}$, $\mu_y$ is supported on the compact set $\tilde{\pi}^{-1}(\{y\})$, in the sense that $\mu_y(E)=0$ whenever $E$ is a measurable set disjoint from $\tilde{\pi}^{-1}(\{y\})$ (this conclusion does \emph{not} require the fibers $\tilde{\pi}^{-1}(\{y\})$ to be measurable). 
Moreover, we have 
\begin{equation}\label{eq-intertwining}
\mu_{\tilde{S}^\gamma(y)}=\tilde{T}^\gamma_*\mu_y
\end{equation}
for all $y\in \tilde{Y}$ and $\gamma\in\Gamma$. 
\end{theorem}

We can apply the canonical disintegration to construct relatively independent products in the dynamical category $\OpProbAlgG$ (cf. \cite[\S 2]{ellis}). 

\begin{theorem}[Relative products in $\OpProbAlgG$]\label{rel-prod}  \cite[\S 8]{jt-foundational} Let $\mathcal{X}_1=(X_1,\mu_1,T_1)$, $\mathcal{X}_2=(X_2,\mu_2,T_2)$ and $\mathcal{Y}=(Y,\nu,S)$ be $\OpProbAlgG$-systems. 
Suppose that one has $\OpProbAlgG$-factors $\pi_1 \colon \mathcal{X}_1 \to \mathcal{Y}$, $\pi_2 \colon \mathcal{X}_2 \to \mathcal{Y}$.  Then there exists a $\OpProbAlgG$-commutative diagram
\begin{center}
\begin{tikzcd}
    & \mathcal{X}_1 \times_\mathcal{Y} \mathcal{X}_2 \arrow[dl, "\Pi_1"'] \arrow[dr,"\Pi_2"] \\ \mathcal{X}_1 \arrow[dr,"\pi_1"'] && \mathcal{X}_2 \arrow[dl,"\pi_2"] \\
    & \mathcal{Y}
\end{tikzcd}
\end{center}
for some $\OpProbAlgG$-system $\mathcal{X}_1 \times_\mathcal{Y} \mathcal{X}_2=(X_1\times_Y X_2,\mu\times_Y\mu, T\times T)$ and $\OpProbAlgG$-morphisms $\Pi_1 \colon \mathcal{X}_1 \times_\mathcal{Y} \mathcal{X}_2 \to \mathcal{X}_1$, $\Pi_2 \colon \mathcal{X}_1 \times_\mathcal{Y} \mathcal{X}_2 \to \mathcal{X}_2$ such that one has
\begin{equation}\label{f1f2}
 \int_{X_1 \times_Y X_2} f_1 f_2 d\mu\times_Y\mu= \int_Y \E(f_1|Y) \E(f_2|Y) d\nu 
\end{equation}
for all $f_1 \in \Linfty(X_1)$, $f_2 \in \Linfty(X_2)$, where we embed $\Linfty(Y)$ into $\Linfty(X_1), \Linfty(X_2)$, and embed these algebras in turn into $\Linfty(X_1 \times_Y X_2)$. Furthermore, the $\SigmaAlg$-algebra of $X_1\times_Y X_2$ is generated by the $\SigmaAlg$-algebras $X_1$, $X_2$ (where we identify the latter with $\sigma$-complete subalgebras of the former in the obvious way).
\end{theorem}

\subsection{The invariant factor functor and the Alaoglu-Birkhoff ergodic theorem}\label{sec-ergodic-thm}

Since we assume no structure on $\Gamma$ other than it being a group, the only ergodic theorem available to us is the Alaoglu-Birkhoff ergodic theorem. 
We first introduce the functor $\Inv$ from $\OpProbAlgG$ into itself:

\begin{definition}[Invariant factor functor]\label{inv-def}  
\text{}
\begin{itemize}
    \item[(i)]  If $\mathcal{X} = (X,\mu,T)$ is a $\OpProbAlgG$-system, we define $$\Inv(\mathcal{X})=(\Inv(X),\mu_{\Inv(X)},T_{\Inv(X)})$$ to be the $\OpProbAlgG$-system with $\AbsMes$-space 
    $$\Inv(X)\coloneqq \{ E \in X: T^\gamma(E) = E \text{ for all } \gamma \in \Gamma \},$$
    measure
    $$ \mu_{\Inv(X)}(E) := \mu(E)$$
    for all $E \in X$, and action defined by
    $$ T_{\Inv(X)}^\gamma(E) \coloneqq T^\gamma(E)$$
    for all $E \in \Inv(X)$ and $\gamma \in \Gamma$.  
    \item[(ii)]  If $f \colon \mathcal{X} \to \mathcal{Y}$ is a $\OpProbAlgG$-factor, we define $\Inv(f) \colon \Inv(\mathcal{X}) \to \Inv(\mathcal{Y})$ by 
    $$\Inv(f)(E) := f(E)$$
    whenever $E \in \Inv(Y)$.
    \item[(iii)] A $\OpProbAlgG$-system $\mathcal{X}$ is said to be \emph{ergodic} if $\Inv(X)$ is the trivial algebra $\{0,1\}$. 
\end{itemize}
\end{definition}

There is a natural epimorphism from the identity functor on $\OpProbAlgG$ to $\Inv$ that gives a $\OpProbAlgG$-factor $\pi \colon \mathcal{X} \to \Inv(\mathcal{X})$ defined by inclusion.  Using this factor, one can view $L^2(\Inv(X))$ as a subspace of $L^2(X)$.  Each shift $T^\gamma \colon X \to X$ induces a unitary Koopman operator $(T^\gamma)^* \colon L^2(X) \to L^2(X)$. For any $f\in L^0(X)$, we denote by $\orb(f)$ the orbit $\{(T^\gamma)^* f \colon \gamma \in \Gamma \}$ of $f$. 

\begin{theorem}[Alaoglu-Birkhoff abstract ergodic theorem]\label{ab-thm} \cite{ab40abstract} Let $\mathcal{X} = (X,\mu, T)$ be a $\OpProbAlgG$-system and $\E(\cdot|\Inv(X)) \colon L^2(X) \to L^2(\Inv(X))$ be the orthogonal projection.  Then for any $f \in L^2(X)$, $\E(f|\Inv(X))$ is the unique element of minimal norm in the closed convex hull of $\orb(f)$.
\end{theorem}

\section{Conditional Hilbert spaces and Hilbert-Schmidt operators}\label{sec-condanal}

In this section, we develop some conditional analysis which is used in Section \ref{sec-compact} to study relatively compact extensions. 
Specifically, we introduce a conditional Hilbert space and its conditional tensor product and establish some useful properties of them. 
We do not aim to give the most general and systematic treatment of conditional Hilbert space theory, 
but present, in a self-contained manner, the minimal amount of that theory needed, directed towards its application to the study of extensions of $\OpProbAlgG$-systems. 
In several remarks, we discuss and provide references how to interpret the results in Boolean sheaf topoi, which are addressed to the reader with familiarity to topos theory and categorical logic.  
The ergodic-theoretically inclined reader, can view the theory of conditional Hilbert spaces as a natural extension of the theory of measurable Hilbert bundles (e.g., see \cite[\S 9]{glasner2015ergodic} and \cite[\S 6]{furstenberg2014recurrence}) in situations where separability and countability hypotheses are not satisfied. 

We do not assume any dynamical structure in this section and fix a $\OpProbAlg$-morphism $\pi:(X,\mu)\to (Y,\nu)$ throughout.  
Recall that we have corresponding embeddings $\pi^*:L^0(Y)\to L^0(X)$ and $\pi^*:L^2(Y)\to L^2(X)$ defined by $\pi^*(f)=f\circ \tilde\pi$ respectively where $\tilde{\pi}:\tilde X\to \tilde Y$ is the canonical representation.  
We identify $L^0(Y)$ with the closed subspace $\pi^*(L^0(Y))$ of $L^0(X)$ (with respect to convergence in probability) and $L^2(Y)$ with the closed subspace $\pi^*(L^2(Y))$ of $L^2(X)$.  
We also identify $\Baire(\tilde{Y})$ with the sub-$\sigma$-algebra $\tilde\pi^*(\Baire(\tilde{Y}))$ of $\Baire(\tilde{X})$.  
Throughout, all equalities, inequalities and inclusions between measurable functions and measurable sets are understood in the almost sure sense. 
We have the canonical disintegration $(\mu_y)_{y\in \tilde{Y}}$ of $\mu$ over $Y$. 
This leads to the relatively independent product $(X\times_Y X, \mu\times_Y \mu)$. 
We can use the canonical $\OpProbAlg$-morphism $\psi:(X\times_Y X, \mu\times_Y \mu)\to (Y,\nu)$ to identify $L^0(Y)$ with a subspace of $L^0(X\times_Y X)$, and similarly for the respective $L^2$ spaces.    

A partition $\mathcal{P}$ of $\tilde{X}$ is called a \emph{$Y$-partition} if its elements are measurable sets in $\Baire(\tilde Y)$. For technical reasons, we allow the empty set to be an element of a partition.  
Let $\mathcal{P}$ be a $Y$-partition and $f_E\in L^0(X)$ for each $E\in \mathcal{P}$. 
Then there is a unique $f\in L^0(X)$ such that $f1_E=f_E1_E$ for all $E\in \mathcal{P}$. 
We denote this unique element $f$ by $\sum_{E\in\mathcal{P}} f_E 1_E$ and call it a \emph{$Y$-countable concatenation} in $L^0(X)$. 
Similarly, we define $Y$-countable concatenations in $L^0(X\times_Y X)$.  
We will be interested in subspaces of $L^0(X)$ and $L^0(X\times_Y X)$ which are closed under $Y$-countable concatenations.  
More specifically, these subspaces are a conditional version of the classical $L^2$ space and its conditional Hilbert space tensor product which are defined as follows. 
 \begin{definition}
We define the \emph{conditional Hilbert space}  
\[
\cltwo\coloneqq\left\{f\in L^0(X): \E(|f|^2|Y)=\int_{\tilde{X}} |f|^2d\mu_y<\infty\right\} 
\]
and the \emph{conditional Hilbert-Schmidt space}
\[
\chs\coloneqq\left\{K\in L^0(X\times_Y X): \E(|K|^2|Y)=\int_{\tilde{X}\times\tilde{X}} |K|^2 d\mu_y\times \mu_y<\infty\right\}.
\]
\end{definition}  
One easily verifies that 
\begin{align*}
\cltwo&=\{f\in L^0(X): \exists g\in L^0(Y), h\in L^2(X) \text{ s.t. }f=g h\}, \\
\chs&=\{f\in L^0(X\times_Y X): \exists g\in L^0(Y), h\in L^2(X\times_Y X) \text{ s.t. }f=g h\}. 
\end{align*}
In particular, both $\cltwo$ and $\chs$ are modules over the ring $L^0(Y)$ closed under $Y$-countable concatenations. 

\begin{remark}
In \cite[\S 2.13]{tao2009poincare}, the same notation $\cltwo$ is used for the smaller $L^\infty(Y)$-module $$\left\{f\in L^2(X): \|\E(|f|^2|Y)\|_{L^\infty(Y)}<\infty\right\}.$$  The crucial difference between the smaller conditional Hilbert space in \cite{tao2009poincare} and our larger conditional Hilbert space $\cltwo$ is that the latter is closed under $Y$-countable concatenations. It is exactly this property that makes $\cltwo$ amenable to topos-theoretic tools, see the following Remark \ref{rem-topos}. 
\end{remark}

We define a conditional version of inner product and norm: 
\begin{definition}\label{def-innerproduct}
For $f,g\in \cltwo$, we define their \emph{conditional inner product} as
\begin{equation}\label{eq-inprod}
\cip{f}{g} \coloneqq \E(f \bar g|Y). 
\end{equation}
For $f\in \cltwo$, we define its \emph{conditional norm}  as
\begin{equation}\label{eq-norm}
\cnorm{f} \coloneqq \sqrt{\cip{f}{f}}. 
\end{equation}
The conditional norm $\cnorm{\cdot}$ induces a probabilistic metric on $\cltwo$ by
$$\pmet(f,g)\coloneqq \int_{Y} \min(1,\cnorm{f-g})d\nu.$$
Similarly, we define $\langle\cdot,\cdot\rangle_{\chs}$, $\hsnorm{\cdot}$, and $\pmeths$. 
\end{definition}
We verify some basic properties.  
\begin{proposition}\label{prop-cnorm-cip}
\begin{itemize}
    \item[(i)] The conditional inner product $\cip{\cdot}{\cdot}$ satisfies the following properties. 
    \begin{itemize}
    \item[(a)] For all $Y$-partitions $\mathcal{P},\mathcal{P}'$ and families $(f_E)_{E\in\mathcal{P}},(f_{E'})_{E'\in\mathcal{P}'}$ in $\cltwo$, we have $$\cip{\sum_{E\in \mathcal{P}}f_E 1_E}{\sum_{E'\in \mathcal{P}'}f_{E'}1_{E'}}=\sum_{E\cap E'\in \mathcal{P}\cap \mathcal{P}'}\cip{f_E}{f_{E'}}1_{E\cap E'},$$
    where $\mathcal{P}\cap \mathcal{P}'\coloneqq \{E\cap E'\colon E\in\mathcal{P}, \ E'\in\mathcal{P}'\}$. 
    \item[(b)] For all $f,g\in \cltwo$, it holds that $\cip{f}{g} = \overline{\cip{g}{f}}$. 
    \item[(c)] For all $f,g,h\in \cltwo$ and $a\in L^0(Y)$, it holds that $$\cip{a f + g}{h}=a \cip{f}{h} + \cip{g}{h}.$$  
    \item[(d)] For all $f\in \cltwo$, it holds that $\cip{f}{f}\geq 0$. 
    \end{itemize}
    Similarly, for $\langle\cdot,\cdot\rangle_{\chs}$. 
    \item[(ii)] The conditional norm $\cnorm{\cdot}$ satisfies the following properties. 
    \begin{itemize}
        \item[(a)] For all $Y$-partitions and families $(f_E)_{E\in\mathcal{P}}$ in $\cltwo$, we have 
        \[
        \cnorm{\sum_{E\in\mathcal{P}} f_E1_E}=\sum_{E\in\mathcal{P}} \cnorm{f_E}1_E.
        \]
        \item[(b)] $\cnorm{f}=0$ if and only if $f=0$.  
        \item[(c)] For every $a\in L^0(Y)$ and $f\in \cltwo$, we have $\cnorm{af}=|a|\cnorm{f}$. 
        \item[(d)] For all $f,g\in \cltwo$, we have $\cnorm{f+g} \leq \cnorm{f}+\cnorm{g}$. 
        \end{itemize}
        Similarly, for $\hsnorm{\cdot}$.  
    \item[(iii)] (Conditional Cauchy--Schwarz inequality) For all $f,g\in \cltwo$, it holds that 
    \begin{equation}\label{eq-cs}
    |\cip{f}{g}|\leq \cnorm{f} \cnorm{g}.  
    \end{equation}
    Similarly, for $\chs$. 
    \item[(iv)] The space 
    \[
    \mathcal{D}=\left\{\sum_{i=1}^n c_i f_i\otimes g_i\colon n\in \N, c_i\in \C, f_i,g_i\in \cltwo\right\},
    \]
    where $f\otimes g(x,x')\coloneqq f(x)g(x')$, is dense in $\chs$ with respect to the metric $\pmeths$.
    \item[(v)] For any $K\in \chs$ and $f\in \cltwo$, define 
    \[
    K\ast_Y f(x)\coloneqq \cip{K(x,\cdot)}{\bar{f}}(\tilde{\pi}(x)). 
    \]
    The mapping $f\mapsto K\ast_Y f$ is well-defined from $\cltwo$ to $\cltwo$. 
   Moreover, it holds 
    \begin{equation}\label{eq-normdom}
    \cnorm{K\ast_Y f}\leq \hsnorm{K} \cnorm{f}
    \end{equation}
    for all $K\in \chs$ and $f\in \cltwo$. 
    In particular,  $f\mapsto K\ast_Y f$ from $\cltwo$ to $\cltwo$ is continuous with respect to the metric $\pmet$. 
   \item[(vi)] Convergence with respect to $L^2$-metric is equivalent to convergence with respect to metric $\pmet$ on $L^2(X)$. 
   Similarly for $\pmeths$ on $L^2(X\times_YX)$.
   \item[(vii)] The space $L^\infty(X)$ (resp.~$L^\infty(X\times_Y X)$) is dense in $\cltwo$ (resp.~$\chs$) with respect to the topology induced by the metric $\pmet$ (resp.~$\pmeths$).  
\end{itemize}
\end{proposition}

\begin{proof}
The assertions in (i)(a)--(d) and (ii)(a)--(c) follow by standard properties of conditional expectations. The conditional triangle inequality in (ii)(d) is implied by the conditional Cauchy--Schwarz inequality in (iii) as in the classical situation.  

As for (iii), we can follow a proof of the classical Cauchy--Schwarz inequality with some necessary modifications. First suppose that $f,g\in \cltwo$ satisfy $\cnorm{f},\cnorm{g},\cnorm{f-ag}>0$ for all $a\in L^0(Y)$. Then we have 
\begin{align*}
0\leq \cnorm{f+ag}^2 =\cnorm{f}^2 + a\cip{g}{f} + \bar{a}\cip{f}{g} + |a|^2\cnorm{g}^2 
\end{align*}
for all $a\in L^0(Y)$. Setting $a=c\frac{\cip{f}{g}}{\cnorm{g}}$ where $c\in L^0(Y)$ satisfies $|c|=1$ and $c\cip{f}{g}=|\cip{f}{g}|$, and after some elementary algebraic manipulations, we obtain \eqref{eq-cs} in this case. 

Second suppose that $f,g\in \cltwo$ satisfy only $\cnorm{f},\cnorm{g}>0$.  Let  
\[
\mathcal{E}=\{E\in \Baire(\tilde{Y}): \exists a\in L^0(Y), f=ag \text{ on } E\}.
\]
By completeness of $Y$ (see Remark \ref{rem-ccc}), there exists a least upper bound $E_{\max}$ for $\mathcal{E}$ (with respect to almost sure inclusion).  
By the countable chain condition, there is $\mathcal{E}'\subset \mathcal{E}$ countable such that $E_{\max}=\bigcup\mathcal{E}'$. 
We can assume that $\mathcal{E}'$ is a $Y$-partition of $E_{\max}$. 
Choose $a\in L^0(Y)$ such that $a=a_E$ on $E$ with $a_E$ such that $f=a_E g$ for all $E\in \mathcal{E}'$. 
Then $f=ag$ on $E_{\max}$.  
It holds that $\cnorm{f},\cnorm{g},\cnorm{f-ag}>0$ on $E_{\max}^c$ for all $a\in L^0(Y)$. 
Thus \eqref{eq-cs} is satisfied on $E_{\max}^c$ by repeating the first step on $E_{\max}^c$. 
Since \eqref{eq-cs} is also trivially satisfied on $E_{\max}$, using Properties (i)(c) and (ii)(c), we obtain \eqref{eq-cs} also in this case. 

Finally suppose that $\cnorm{f}$ or $\cnorm{g}$ are equal to $0$ on a set of positive measure. 
Let $E=\{\cnorm{f}=0\}\cup \{ \cnorm{g}=0\}$. 
Then $\eqref{eq-cs}$ is satisfied on $E^c$ by repeating the second step on $E^c$. 
Since \eqref{eq-cs} is also trivially satisfied on $E$, we obtain \eqref{eq-cs} also in this case, and this completes the proof of (iii). 

It is enough to prove (iv) for $f\in \chs$ with $f\geq 0$.  
By the Stone-Weierstra\ss \, theorem, the algebra of finite disjoint unions of rectangles $E\times F$ with $E,F\in\Baire(\tilde{X})$ generates $\Baire(\tilde{X}\times \tilde{X})$. 
Hence there is a sequence $(K_n)$ in $\mathcal{D}$ such that $K_n\uparrow K$ $\tilde\mu\times_{\tilde Y}\tilde\mu$-almost surely as $n\to \infty$.  
By monotone convergence, $\hsnorm{K_n- K}\to 0$ $\tilde\nu$-almost surely. 
Hence $\hsnorm{K_n-K}\to 0$ in convergence in $\tilde\nu$-measure, and this proves (iv). 

As for (v), let $K\in \chs$ and $f\in \cltwo$. 
By the conditional Cauchy--Schwarz inequality and the Fubini--Tonelli theorem, one obtains
\begin{align*}
    \int_{\tilde{X}}\left| K\ast_Y f \right|^2(x) d\mu_y(x) &= \int_{\tilde{X}}\left|\cip{K(x,\cdot)}{\bar{f}}(\tilde{\pi}(x)) \right|^2 d\mu_y(x) \\
    &\leq \cnorm{f}^2\int_{\tilde{X}} \cnorm{K(x,\cdot)}^2(\tilde{\pi}(x)) d\mu_y(x) \\
    &=\hsnorm{K}^2 \cnorm{f}^2 <\infty
\end{align*}
This shows that $K\ast_Y f\in \cltwo$, \eqref{eq-normdom} and continuity.

As for (vi), we only prove the claim for $\cltwo$ (the arguments for $\chs$ are identical). 
For $f\in L^2(X)$, it holds 
\begin{equation}\label{eq-tower1}
\|f\|^2_{L^2(X)}=\int_{Y}\cnorm{f}^2 d\nu. 
\end{equation}
The assertion follows from the fact that $L^2$-convergence implies almost sure convergence upon passing to a subsequence, almost sure convergence implies convergence in measure, and that $L^2$-convergence and convergence in measure are equivalent. 

Finally as for (vii), let $f\in \cltwo$. Define $f_m=f 1_{\cnorm{f}\leq m}$ for $m\geq 1$. Then $\pmet(f,f_m)\to 0$ as $m\to \infty$. 
By (ii)(c) and \eqref{eq-tower1}, $f_m\in L^2(X)$. 
As $L^\infty(X)$ is dense in $L^2(X)$ in $L^2$-topology, we can approximate each $f_m$ by a sequence $(f_{m,n})$ in $L^\infty(X)$ in $L^2$-metric. The claim now follows from \eqref{eq-tower1}, (vi), and triangle inequality. 
\end{proof}

\begin{remark}\label{rem-topos}
Let $\mathtt{Sh}(Y)$ be the Grothendieck topos of sheaves on the site $(Y,J)$ where $J$ is the Grothendieck sup-topology on the complete Boolean algebra $Y$. 
We refer the interested reader to \cite{maclane2012sheaves} for an introduction to Grothendieck sheaf topoi. 
By the countable chain condition (see Remark \ref{rem-ccc}), a Grothendieck basis of the sup-topology $J$ is given by the sets of countable partitions of elements of $Y$.  
Since the spaces $L^0(Y), \cltwo, \chs$ are closed under countable $Y$-concatenations, they can be given the structure of a sheaf on the site $(Y,J)$.  
Since $Y$ is a complete Boolean algebra, the topos $\mathtt{Sh}(Y)$ is Boolean and satisfies the axiom of choice. 
Thus its internal logic is strong enough to permit an internal mathematical discourse sufficient for the bulk of classical mathematics, see \cite[Chatper IV]{maclane2012sheaves}. 
In particular, we can interpret the external space $L^0(Y)$ as the internal complex numbers of $\mathtt{Sh}(Y)$.  
Now Proposition \ref{prop-cnorm-cip} shows that the external space $\cltwo$ can be interpreted as an internal complex Hilbert space, and the external space $\chs$ as its internal Hilbert space tensor product whose elements represent the internal Hilbert-Schmidt operators.  
Externally $L^0(Y)$ is a commutative ring and $\cltwo$ and $\chs$ are topological modules over $L^0(Y)$ (given the topology induced by $\pmet$, $\pmeths$ and convergence in probability).   

The conditional analysis of this section is an external interpretation and simplification of some results in the internal Hilbert space theory of $\mathtt{Sh}(Y)$. 
This interpretation is enabled by conditional set theory and conditional analysis as developed in \cite{filipovic2009separation,cheridito2015conditional,drapeau2016algebra}, which, initially, were developed  independently of topos theory and can be viewed in retrospect as an exploration of the semantics of the internal discourse of Boolean Grothendieck topoi. (The connection of conditional set theory to Boolean Grothendieck topoi is established in \cite{jamneshan2014topos}, see also \cite{carl2018transfer} for a relation to the akin theory of Boolean-valued models.) 
\end{remark}

We continue to develop some basic conditional analysis of the conditional Hilbert space $\cltwo$. 

\begin{definition}
An \emph{$L^0(Y)$-linear combination} in $\cltwo$ is an expression of the form $\sum_{f\in\mathcal{F}} a_f f$ for some finite set $\mathcal{F}$ in $\cltwo$ and $a_f\in L^0(Y), f\in\mathcal{F}$. 
A subset $\mathcal{M}\subset \cltwo$ is called a \emph{finitely generated $L^0(Y)$-submodule} if there is a finite collection $\mathcal{F}$ in $\cltwo$ such that each $f\in \mathcal{M}$ can be expressed as $\sum_{f\in\mathcal{F}} a_f f$ for some $(a_f)_{f\in\mathcal{F}}$ in $L^0(Y)$. 
\end{definition}

\begin{proposition}[Conditional Gram-Schmidt process]\label{prop-GSP}
Let $\mathcal{M}$ be a finitely generated $L^0(Y)$-submodule of $\cltwo$. 
Then there is a finite $Y$-partition $\mathcal{P}$ and for each $E\in\mathcal{P}$ there is a finite set $\mathcal{F}_E$ in $\cltwo$ satisfying the following properties.   
\begin{itemize}
    \item[(i)] There is $E_0\in\mathcal{P}$ such that $\mathcal{F}_{E_0}=\{0\}$. 
    \item[(ii)] $\cnorm{f}=1$ on $E$ for all $f\in \mathcal{F}_E$ and $E\in\mathcal{P}\backslash\{E_0\}$. 
    \item[(iii)] $\cip{f}{g}=0$ on $E$ for all distinct $f,g\in \mathcal{F}_E$ and $E\in\mathcal{P}$. 
    \item[(iv)] It holds that 
    $$\mathcal{M}=\left\{\sum_{E\in\mathcal{P}} f_E 1_E \colon \forall\, E\in\mathcal{P} \, \exists \, (a^E_f)_{f\in\mathcal{F}_E} \subset L^0(Y) \text{ s.t. } f_E=\sum_{f\in \mathcal{F}_E} a^E_f f\right\}.$$
\end{itemize}
We call $\cdim(\mathcal{M})=\sum_{E\in\mathcal{P}} \#(\mathcal{F}_E)|E$ the \emph{conditional dimension} of $\mathcal{M}$, where $\#$ denotes cardinality (note that $\cdim(\mathcal{M})$ is a $\Baire(\tilde{Y})$-measurable random variable). 
\end{proposition}

\begin{proof}
Let $f_1,\ldots,f_n\in \cltwo$ be a collection of generators of $\mathcal{M}$.  
Define $$h_1=\frac{f_1}{\cnorm{f_1}} 1_{\cnorm{f_1}>0}.$$ 
Next suppose we defined $h_1,\ldots,h_k$ with $k\leq n-1$. 
Set $$g_{k+1}=f_{k+1}-\sum_{i=1}^k \cip{f_{k+1}}{h_i} h_i$$ and define 
\[
h_{k+1}=\frac{g_{k+1}}{\cnorm{g_{k+1}}} 1_{ \cnorm{g_{k+1}}>0}.
\]
Denote by $F_i=\{ \cnorm{g_{i}}>0\}$ for $i=1,\ldots,n$.  
Form all finite intersections $E_{1}\cap E_{2} \cap \ldots \cap E_{n}$ where $E_i=F_i$ or $E_i=F_i^c$ for each $i=1,\ldots,n$. Let $\mathcal{P}_0$ denote the set of all such finite intersections whose measure is positive. Then $\mathcal{P}_0$ forms a $Y$-partition. For $E=E_{1}\cap E_{2} \cap \ldots \cap E_{n}\in \mathcal{P}_0$, let $\mathcal{F}_{E}=\{h_{i}: A_i=F_i, \, i=1,\ldots,n\}$.  
Let $\mathcal{P}$ consists of those $E\in \mathcal{P}_0$ such that $\mathcal{F}_E$ is not empty. 
Set $E_0=(\bigcup_{E\in\mathcal{P}} E)^c$. 
By construction, $\mathcal{P}$, $(\mathcal{F}_E)_{E\in\mathcal{P}}$, and $E_0$ satisfy properties (i)--(iv).  
\end{proof}

\begin{lemma}\label{lem-finite-closed}
Let $\mathcal{M}$ be a finitely generated $L^0(Y)$-submodule of $\cltwo$. 
Then $\mathcal{M}$ is closed with respect to the metric $\pmet$. 
\end{lemma}

\begin{proof}
Let $\mathcal{P}$ and $(\mathcal{F}_E)_{E\in\mathcal{P}}$ satisfy properties (i)-(iv) in Lemma \ref{prop-GSP} for $\mathcal{M}$. 
Let $(f_n)$ be a sequence in $\mathcal{M}$ such that $\pmet(f_n,f)\to 0$ as $n\to \infty$ for some $f\in \cltwo$. 
Each $f_n$ is of the form $$f_n=\sum_{E\in \mathcal{P}}\left(\sum_{f_E\in \mathcal{F}_E}a_{f_E}^n f_E \right)|E$$ for some $(a_{f_E}^n)_{f\in\mathcal{F}_E}$ for each $E\in\mathcal{P}$. 
For arbitrary $n,m$, by Proposition \ref{prop-cnorm-cip}(ii) and Proposition \ref{prop-GSP}(ii) \& (iii),
\begin{equation}\label{eq-cauchy}
    \cnorm{f_n-f_m}=\sum_{E\in \mathcal{P}}\left(\sum_{f_E\in \mathcal{F}_E}|a_{f_E}^n-a_{f_E}^m|\right)|E.
\end{equation}
Since $\pmet(f_n,f_m)\to 0$ as $n,m\to \infty$ by hypothesis, the claim follows by applying \cite[Lemma  4.6]{Kallenberg2002} to \eqref{eq-cauchy}.  
\end{proof}

The following results are needed in the characterization of compact extensions in the next section to achieve a conditional spectral decomposition for conditional Hilbert--Schmidt operators. Due to our use of the Borel functional calculus, we treat conditional Hilbert--Schmidt operators on the classical Hilbert space $L^2(X)$ which we however view as an $L^\infty(Y)$-module by multiplication. 

\begin{definition}[Conditional orthonormal basis]
Let $\mathcal{M}$ be $L^2(X)$-closed $L^\infty(Y)$-submodule of $L^2(X)$.   
A subset $M$ of $\mathcal{M}$ is said to be a \emph{conditional orthonormal basis} if the following properties are satisfied. 
\begin{itemize}
    \item[(i)] $\cip{f}{g}=0$ for all $f,g\in M$. 
    \item[(ii)] $\cip{f}{f}=1_E$ for some $E\in Y$ for all $f\in M$ (where $E$ may depend on $f$). 
    \item[(iii)] $\mathcal{M}=\bigoplus_{f\in M} \overline{L^\infty(Y) f}$. 
\end{itemize}  
\end{definition}

\begin{proposition}\label{prop-ppbasis}
    Any $L^2(X)$-closed $L^\infty(Y)$-submodule of $L^2(X)$ has a conditional orthonormal basis. 
\end{proposition}

\begin{proof}
    This is a commutative special case of the existence of Pimsner--Popa orthonormal basis in right modules over tracial von Neumann algebras, see \cite[Proposition 8.4.11]{AP17}. 
\end{proof}

\begin{proposition}\label{prop-chs}
    Let $\mathcal{M}$ be $L^2(X)$-closed $L^\infty(Y)$-submodule of $L^2(X)$, and let $K\in L^\infty(X\times_Y X)$. 
    Then $K\ast_Y\colon \mathcal{M}\to L^2(X)$ is a well-defined $L^\infty(Y)$-linear classical bounded operator. Moreover, it is \emph{conditionally Hilbert--Schmidt} in the sense that for any conditional orthonormal basis $M$ of $\mathcal{M}$, 
    \[
    \sup_{F\subset M \text{ finite}}\sum_{f\in F} \cnorm{K\ast_Y f}^2<\infty, 
    \]
    where the (essential) supremum of the measurable functions on the left-hand side exists by completeness of $Y$ (see Remark \ref{rem-ccc}), and \emph{a priori} is measurable with values in $[0,\infty]$. 
\end{proposition}

\begin{proof}
      Suppose $\|K\|_{L^\infty(X\times_Y X)}=C$ for some constant $C>0$.
      By \eqref{eq-normdom} and \eqref{eq-tower1}, 
    \begin{align*}
        \|K\ast_Y f\|_{L^2(X)}^2 &= \int_Y \cnorm{K\ast_Y f}^2 d\nu \\
        &\leq \int_Y C^2 \cnorm{f}^2 d\nu \\ 
        &= C^2 \|f\|_{L^2(X)}^2.
    \end{align*}
    This proves the first claim.

As for the second claim, let $M$ be a conditional orthonormal basis of $\mathcal{M}$ and let $F\subset M$ be a finite subset of $M$. 
Applying Bessel's inequality pointwise almost everywhere, we have $$\sum_{f\in F} |(K\ast_Y f)(x)|^2\leq \cnorm{K(x,\cdot)}^2(\tilde{\pi}(x))<C^2,$$
The claim follows from the monotone convergence theorem for conditional expectations since the essential supremum is attained by countable subfamily due to the countable chain condition (see Remark \ref{rem-ccc}) and since the family $\sum_{f\in M} \cnorm{K\ast_Y f}^2$ parametrized by finite subsets of $M$ is directed upwards.  
\end{proof}

We have the following criterion to decide when an $L^2(X)$-closed $L^\infty(Y)$-submodule of $L^2(X)$ is finitely generated. 

\begin{proposition}\label{prop-finitegen}
    Let $\mathcal{M}$ be $L^2(X)$-closed $L^\infty(Y)$-submodule of $L^2(X)$. $\mathcal{M}$ is finitely generated if and only if there is a constant $C$ such that for every conditional orthonormal basis $M$ in $\mathcal{M}$ it holds that $ \sup_{F\subset M \text{ finite}}\sum_{f\in F} \cnorm{f}^2<C$.
\end{proposition}

\begin{proof}
    This is the special case of \cite[Proposition 9.3.2 (i)]{AP17} in the setting of commutative tracial von Neumann algebras, once one observes that $\hat{E}_Z(1)$ as defined in that proposition equals $\sup_{F\subset M \text{ finite}}\sum_{f\in F} \cnorm{f}^2$ by   \cite[Lemma~8.4.8]{AP17} and the observations at the beginning of \cite[\S 9.3]{AP17}.
\end{proof}

\section{Relatively compact extensions}\label{sec-compact}

In this section, we establish various characterizations of relatively compact extensions of $\OpProbAlgG$-systems which are collected in

\begin{theorem}\label{thm-compact}
Let $\mathcal{X}=(X,\mu,T)$ and $\mathcal{Y}= (Y,\nu,S)$ be $\OpProbAlgG$-systems and $\pi:\mathcal{X}\to \mathcal{Y}$ be a $\OpProbAlgG$-extension. 
Then the following are equivalent. 
\begin{itemize}
    \item[(i)] The conditional Hilbert space $\cltwo$ is the $\pmet$-closure of the $L^0(Y)$-linear span of 
    \[
    \{K\ast_Y f: K\in \chs \text{ $\Gamma$-invariant, } f\in \cltwo \}. 
    \]
    \item[(ii)] The conditional Hilbert space $\cltwo$ is the $\pmet$-closure of the union of all its finitely generated and $\Gamma$-invariant $L^0(Y)$-submodules. 
    \item[(iii)] There exists a dense set $\mathcal{G}$ in $\cltwo$ with respect to the metric $\pmet$ such that for all $f\in \mathcal{G}$ and every $\eps>0$ there is a finite set $\mathcal{F}$ in $\cltwo$ such that for all $\gamma\in \Gamma$,
      \[
    \min_{g\in\mathcal{F}} \cnorm{(T^\gamma)^*f - g}<\eps. 
    \]
    \item[(i)'] The classical Hilbert space $L^2(X)$ is the $L^2$-closure of the $\C$-linear span of
    \[
    \{K\ast_Y f: K\in L^\infty(X\times_Y X) \text{ $\Gamma$-invariant, } f\in L^2(X) \}.
    \]
    \item[(ii)'] The classical Hilbert space $L^2(X)$ is the $L^2$-closure of the union of all its $\Gamma$-invariant and finitely generated $L^\infty(Y)$-submodules.  
    \item[(iii)'] There exists a dense set $\mathcal{H}$ in $L^2(X)$ such that for all $f\in\mathcal{H}$ and every $\eps>0$ there is a finite set $\mathcal{F}$ in $L^2(X)$ such that for all $\gamma\in\Gamma$,
    \[
    \min_{g\in\mathcal{F}} \cnorm{(T^\gamma)^*f - g}<\eps. 
    \]
\end{itemize}
A $\OpProbAlgG$-morphism $\pi$ fulfilling one (and therefore all) of the above six properties is called a \emph{relatively compact $\OpProbAlgG$-extension}. 
\end{theorem}

\begin{remark}
Property (ii)' is used by Zimmer in \cite{zimmer1976ergodic,zimmer1976extension} under the name of ''relatively discrete spectrum''.  Ellis \cite{ellis} used a variant of property (ii)' to generalize some parts of Zimmer's work, see Section \ref{sec-isometric}.  
Properties (i)' and (iii)' appear in Furstenberg \cite[\S 6]{furstenberg2014recurrence}.  
We prove that (i)', (ii)', and (iii)' are equivalent and relate them to their analogues in the larger conditional Hilbert spaces $\cltwo$ and $\chs$. 
Our proof of Theorem \ref{thm-compact} combines ideas from the countable theory as in \cite[\S 6]{furstenberg2014recurrence} with the conditional analysis developed in Section \ref{sec-condanal}. 
\end{remark}

\begin{remark}\label{rem-internal}
Using conditional integration theory developed in \cite{jamneshan2018measure} (in particular, a conditional version of the Carath\'eodory extension theorem and the results in \cite[\S 4]{jamneshan2018measure}), one can construct, in a slightly tedious process, an internal $\mathtt{Sh}(Y)$ system from an external extension $\pi:\mathcal{X}\to \mathcal{Y}$.  Then properties (i)-(iii) in Theorem \ref{thm-compact} characterizing relatively compact extensions or systems of relative discrete spectrum can be viewed as an external interpretation of characterizations of internal compact systems or internal systems of discrete spectrum.  
\end{remark}

\begin{lemma}\label{lem-comp1}
Assertion (i)' implies assertion (ii)' in Theorem \ref{thm-compact}. 
\end{lemma}

\begin{proof}
Since the finite sum of $\Gamma$-invariant and finitely generated $L^\infty(Y)$-submodules is a $\Gamma$-invariant and finitely generated $L^\infty(Y)$-submodule, it suffices to show that the ranges of $K \ast_Y$, where $K \in L^\infty(X\times_Y X)$ is $\Gamma$-invariant, are contained in the closure of the union of all $\Gamma$-invariant and finitely generated $L^\infty(Y)$-submodules of $L^2(X)$. 

Let $K \in L^\infty(X\times_Y X)$ be $\Gamma$-invariant. By decomposing 
    \begin{align*}
        K(x,y) = \frac{K(x,y) + \overline{K(y,x)}}{2} + \mathrm{i} \frac{K(x,y) - \overline{K(y,x)}}{2\mathrm{i}},
    \end{align*}
we may reduce to the case that $K(x,y) = \overline{K(y,x)}$, and then by Proposition \ref{prop-chs}, \(K \ast_Y \colon L^2(X) \to L^2(X)\) is a bounded self-adjoint operator. Additionally, by treating the positive and negative spectrum separately, we can assume that $K$ is a positive operator. For \(\eps > 0\), consider the spectral projection  
\[
P_\eps \coloneqq 1_{[\eps, \|K \ast_Y \|]}(K \ast_Y).
\]  
By standard properties of spectral projections, we have
\begin{align}\label{propspectral}
P_\eps \circ (K \ast_Y) = (K \ast_Y) \circ P_\eps \geq \eps P_\eps,
\end{align}  
where the inequality means that 
\begin{align*}
\langle K\ast_Y P_\eps f\mid P_\eps f\rangle_{L^2(X)}\geq \eps \langle P_\eps f \mid P_\eps f \rangle_{L^2(X)}
\end{align*}
for all $f\in L^2(X)$. 

Since \(P_\eps \) arises as a limit of polynomials in \(K \ast_Y\) in the strong operator topology, \(P_\eps \) is \(\Gamma\)-equivariant. Additionally, \(P_{\eps} \) is \(L^\infty(Y)\)-linear since \(L^\infty(Y)\)-linearity is preserved when passing to strong operator limits. It follows that $\mathcal{H}_{\eps}\coloneqq P_\eps(L^2(X))$ is a $\Gamma$-invariant $L^\infty(Y)$-submodule of $L^2(X)$. Since $P_\eps$ is an orthogonal projection, $\mathcal{H}_{\eps}$ is also $L^2(X)$-closed. 

We now show that $\mathcal{H}_{\eps}$ is finitely generated. By \eqref{propspectral}, 
\[
\langle K\ast_Y f, f\rangle_{L^2(X)}\geq \eps \langle f,f\rangle_{L^2(X)}
\]
for all $f\in \mathcal{H}_\eps $. We claim that
\begin{align}\label{eq-magic}
\cip{K\ast_Y f}{f} \geq \eps \cip{f}{f}    
\end{align}
for each $f \in \mathcal{H}_{\eps} $. 
Indeed, let $E=\{\cip{K\ast_Y f}{f} < \eps \cip{f}{f}    \}$. Since $\mathcal{H}_{\eps}$ is an $L^\infty(Y)$-module, we have $g\coloneqq 1_E f\in \mathcal{H}_\eps$. 
By \eqref{eq-tower1} and \eqref{propspectral}, 
\begin{multline*}
    0\leq \langle K\ast_Y g, g\rangle_{L^2(X)} - \eps \langle f,f\rangle_{L^2(X)} = \int_Y \cip{K\ast_Y g}{g} - \eps \cip{g}{g} d\nu \\ =    \int_E \cip{K\ast_Y f}{f} - \eps \cip{f}{f} d\nu \leq 0. 
\end{multline*}
Thus $\nu(E)=0$, proving the claim. 

By the conditional Cauchy--Schwarz inequality applied to \eqref{eq-magic}, we have
    \begin{align*}
        \|K\ast_Y f\|_{X\mid Y} \cdot \cnorm{f} \geq \eps \cnorm{f}^2
    \end{align*}
and this implies $\cnorm{f} \leq \frac{1}{\eps} \|K\ast_Y f\|_{X\mid Y}$ for $f \in \mathcal{H}_{\eps} $. Using that $K \ast_Y \colon L^2(X) \rightarrow L^2(X)$ is a conditional Hilbert--Schmidt operator (see Proposition \ref{prop-chs}), we find $C > 0$ such that 
    \begin{align*} 
        \sum_{f \in M} \|K\ast_Y f\|_{X\mid Y}^2 \leq C
    \end{align*}
for every conditionally orthonormal basis $M \subseteq L^2(X)$. In particular, if $M \subseteq \mathcal{H}_{\eps} $ is a conditionally orthonormal basis of $\mathcal{H}_{\eps} $, then 
    \begin{align*}
        \sum_{f \in M} \cnorm{f}^2 \leq  \frac{1}{\eps} \sum_{f \in M} \|K\ast_Y f\|_{X\mid Y}^2\leq \frac{C}{\eps}.
    \end{align*}
Using Proposition \ref{prop-finitegen}, we obtain that the $L^\infty(Y)$-submodules $\mathcal{H}_{\eps} $ are finitely generated.

If $f \in L^2(X)$ we obtain by the properties of spectral projections that 
    \begin{align*}
        K\ast_Y f = \lim_{n \rightarrow \infty} P_{\frac{1}{n}}(K \ast_Y f).
    \end{align*}
Therefore, the image of $K \ast_Y$ is contained in the $L^2(X)$-closure of the union of all $\Gamma$-invariant and finitely generated $L^\infty(Y)$-submodules of $L^2(X)$, this concludes the proof.
\end{proof}

\begin{remark}
An alternative proof of the implication (i)' $\Rightarrow$ (ii)', or better to say of (i) $\Rightarrow$ (ii) (from which then (i)' $\Rightarrow$ (ii)' could be easily deduced), would be to invoke a conditional spectral theorem for the conditional Hilbert-Schmidt operator $K\ast_Y \colon \cltwo \to \cltwo$ for any $K\in \chs$. The internal logic of $\mathtt{Sh}(Y)$ admits such a $\mathtt{Sh}(Y)$-spectral theorem, whose proof one could then interpret externally by hand for the conditional Hilbert-Schmidt operator $K\ast_Y \colon \cltwo \to \cltwo$ (in some sense, this is carried out in a topological setting in \cite[Part I, Section 4]{ehk}, cf. the comments at the end of \cite[Part I]{ehk}). We can also compare with the relative spectral theorem for measurable Hilbert bundles proved in \cite[Chapter 9, Section 3]{glasner2015ergodic} in the countable complexity setting. 
\end{remark}

\begin{lemma}\label{lem-comp2}
Assertion (ii) implies assertion (iii) in Theorem \ref{thm-compact}. 
\end{lemma}

\begin{proof}
Let $\mathfrak{A}$ be the union of finitely generated $\Gamma$-invariant $L^0(Y)$-submodules, and 
let $\mathcal{G}=\{f\in \mathfrak{A}: \exists M>0 \text{ s.t. } \cnorm{f}< M\}$. 
By assumption and Lemma \ref{lem-finite-closed}, $\mathcal{G}$ is dense in $\cltwo$ with respect to the metric $\pmet$. 

Fix $f\in\mathcal{G}$ and $\eps>0$. 
Then there is a finitely generated and $\Gamma$-invariant $L^0(Y)$-submodule $\mathcal{M}$ such that $f\in\mathcal{M}$.  
By Proposition \ref{prop-GSP}, there are a finite $Y$-partition $\mathcal{P}$ and for each $E\in \mathcal{P}$ a finite family $\mathcal{F}_E$ such that for all $f\in \mathcal{M}$ we have the representation $f=\sum_{E\in\mathcal{P}} f_E1_E$ where $f_E=\sum_{g\in \mathcal{F}_E} a_{g} g$ is an $L^0(Y)$-linear combination for each $E\in \mathcal{P}$. 
Since $\cnorm{(T^\gamma)^*f}=(T^\gamma)^*\cnorm{f} <M$ for all $\gamma\in\Gamma$, the orbit $\orb(f)$ is bounded with respect to the conditional norm $\cnorm{\cdot}$. 
It follows from Proposition \ref{prop-cnorm-cip}(ii) and Proposition \ref{prop-GSP} that for all $h\in \orb(f)$ 
\begin{align*}
\cnorm{h}&=\cnorm{\sum_{E\in\mathcal{P}} h_E 1_E}=\sum_{E\in\mathcal{P}}\cnorm{h_E} 1_E\\
&=\sum_{E\in\mathcal{P}} \left(\sum_{g\in\mathcal{F}_E} |a^h_{g}|^{2}\right)^{1/2} 1_E <M. 
\end{align*}
For each $E\in \mathcal{P}$ choose $c_1^E,\ldots, c_{j_E}^E\in \C^{\#(\mathcal{F}_{E})}$ such that the closed ball in $\C^{\#(\mathcal{F}_{E})}$ with radius $M$ and center the origin is covered by the union of the open balls of radius $\eps$ and center $c_k^E$, $k=1,\ldots, j_E$. 
Denote by $C_k^E$ the constant measurable vector in $L^0(Y)^{\#(\mathcal{F}_{E})}$ with value $c_k^E$.  
For $f\in \mathcal{F}_{E}$, let $C_k^E(f)$ denote a component of the vector $C_k^E$. 
Put 
\[
\mathcal{L}_E=\left\{\sum_{f\in\mathcal{F}_{E}} C_k^E(f) f: k=1,\ldots, j_E\right\}
\]
and define 
\begin{equation*}\label{eq-G}
\mathcal{L}= \left\{\sum_{E\in\mathcal{P}} g_E|E: g_E\in\mathcal{L}_E \text{ for each } E\in\mathcal{P} \right\}. 
\end{equation*}
Clearly, $\mathcal{L}$ is finite and by construction, it holds that
    \[
    \min_{g\in \mathcal{L}} \cnorm{g-f}<\eps. 
    \] 
\end{proof}

\begin{remark}
The idea of the previous proof is a conditional version of the Heine-Borel theorem (see \cite[Theorem 4.6]{drapeau2016algebra}) or the internal Heine-Borel theorem of $\mathtt{Sh}(Y)$. 
\end{remark}

\begin{lemma}\label{lem-comp3}
Assertion (i) is equivalent to assertion (i)' in Theorem \ref{thm-compact}. 
\end{lemma}

\begin{proof}
The equivalence is a consequence of Proposition \ref{prop-cnorm-cip} (v)-(vii).  
\end{proof}

\begin{lemma}\label{lem-comp4}
Assertion (ii) is equivalent to assertion (ii)' in Theorem \ref{thm-compact}. 
\end{lemma}

\begin{proof}
We first show that (ii) implies (ii)'. 
Let $\mathcal{M}$ be a finitely generated $\Gamma$-invariant $L^0(Y)$-submodule of $\cltwo$. 
For $\mathcal{M}$ choose a finite $Y$-partition $\mathcal{P}$ and for each $E\in\mathcal{P}$ a finite set $\mathcal{F}_E$ satisfying the properties in Proposition \ref{prop-GSP}. 
Let 
\[
\mathcal{L}=\left\{\sum_{E\in\mathcal{P}} g_E|E: g_E\in \mathcal{F}_E \text{ each } E\in\mathcal{P}\right\}.
\]
By property (ii) in Proposition \ref{prop-GSP},  $\cnorm{g}\leq 1$ for each $g\in\mathcal{L}$. 
Thus by \eqref{eq-tower1}, $\|g\|_{L^2(X)}\leq 1$ for all $g\in\mathcal{L}$. 
Let $\mathcal{M}'$ be the $L^\infty(Y)$-submodule of $L^2(X)$ generated by $\mathcal{L}$. 
    Then $\mathcal{M}'$ is $\Gamma$-invariant since $\mathcal{M}$ is $\Gamma$-invariant, and by Proposition \ref{prop-cnorm-cip}(vi) and Lemma \ref{lem-finite-closed}, $\mathcal{M}'$ is also $L^2$-closed. 
It remains to show that the union of all $\mathcal{M}'$, where $\mathcal{M}$ is a finitely generated and $\Gamma$-invariant $L^0(Y)$-submodule of $\cltwo$, is dense in $L^2(X)$. 
But this follows from Proposition \ref{prop-cnorm-cip}(vii). 

Let us show that (ii)' implies (ii). 
Let $\mathcal{M}'$ be a finitely generated $\Gamma$-invariant $L^\infty(Y)$-submodule of $L^2(X)$. Let $\mathcal{L}$ be a finite set of generators of $\mathcal{M}'$.  
Let $\mathcal{M}=\{\sum_{g\in\mathcal{L}} a_g g: a_g\in L^0(Y) \forall g \in\mathcal{L}\}$. 
By construction, $\mathcal{M}$ is a finitely generated $\Gamma$-invariant  $L^0(Y)$-submodule of $\cltwo$. 
By Proposition \ref{prop-cnorm-cip}(vii), the union of all these $\mathcal{M}$ is dense in $\cltwo$ with respect to $\pmet$. 
\end{proof}

\begin{lemma}\label{lem-comp5}
Assertion (iii) is equivalent to assertion (iii)' in Theorem \ref{thm-compact}. 
\end{lemma}

\begin{proof}
We show that (iii) implies (iii)'. 
By assumption, there is a dense set $\mathcal{G}$ in $\cltwo$ such that the orbit of every $g\in \mathcal{G}$ has the conditional total boundedness property in $\cltwo$. We will construct a dense set $\mathcal{H}$ in $L^2(X)$ such that the orbit of every $f\in \mathcal{H}$ has the same conditional total boundedness property but within $L^2(X)$. 

Let $\eps>0$ and let $0<\delta<1$.  
Let $f\in L^2(X)$ and let $g\in \mathcal{G}$ be such that $\pmet(f,g)<\frac{\delta\eps}{2}$. 
Then the measure of $B= \{\cnorm{f-g}<\frac{\delta}{2}\}$ is greater than $1-\frac{\eps}{2}$. 
There is a finite subset $\mathcal{F}$ of $\cltwo$ such that for all $\gamma\in \Gamma$,
    \[
    \min_{h\in\mathcal{F}} \cnorm{(T^\gamma)^*g - h}<\frac{\delta}{2}. 
    \]
By the conditional triangle inequality, for all $\gamma\in\Gamma$,
    \[
    \min_{h\in\mathcal{F}\cup\{0\}} \cnorm{(T^\gamma)^*(f1_B) - h}<\delta. 
    \] 
Suppose $\mathcal{F}$ has $n$ elements. By monotone convergence, for each $h\in\mathcal{F}$ pick a measurable subset $A_h$ of $Y$ such that $\nu(A_h)\geq 1-\frac{\eps}{2n}$ and $h1_{A_h}\in L^2(X)$. Let $A=\bigcap_{i=1}^n A_h$. 
Then $\nu(A)\geq 1- \frac{\eps}{2}$.  
By the conditional triangle inequality, 
\[
 \min_{h\in\mathcal{F}\cup\{0\}} \cnorm{(T^\gamma)^*(f1_B) - h1_{A_h}}<\delta \quad \text{on } A. 
\]
By the completeness of the probability algebra $Y$ and the countable chain condition (cf. Remark \ref{rem-ccc}), we have 
\[
A^\ast\coloneqq \bigcup_{\gamma\in\Gamma} T^\gamma(A) = \bigcup_{n\in\mathbb{N}} T^{\gamma_n}(A)
\]
for some countable set $\{\gamma_n\}\subset \Gamma$. Using a trick of Furstenberg (cf. proof of \cite[Theorem 6.13, $C_1\Rightarrow C_2$]{furstenberg2014recurrence}) by modifying the $h1_{A_h}$ (while keeping them in $L^2(X)$), we reach that 
\[
 \min_{\tilde{h}\in\tilde{\mathcal{F}}\cup\{0\}} \cnorm{(T^\gamma)^*(f1_B) - \tilde{h}}<\delta \quad \text{on } A^*,  
\]
where $\tilde{\mathcal{F}}$ collects the modified $h1_{A_h}$ for all $h\in\mathcal{F}$. 
Since $A^*$ is $\Gamma$-invariant, we obtain 
\[
 \min_{\tilde{h}\in\tilde{\mathcal{F}}\cup\{0\}} \cnorm{(T^\gamma)^*(f1_{B\cap A^*}) - \tilde{h}1_{A^*}}<\delta
\]
globally and the measure of $B\cap A^*$ is at least $1-\eps$. By Proposition \ref{prop-cnorm-cip}(vii), the collection $\mathcal{H}$ of  $f1_{B\cap A^*}$ as constructed above starting with any $f\in L^2(X)$ is dense in $L^2(X)$. This proves that (iii) implies (iii)'. 

The opposite direction is an immediate consequence of Proposition \ref{prop-cnorm-cip}(vi) and (vii). 
\end{proof}

We conclude the proof of Theorem \ref{thm-compact} by establishing:
\begin{lemma}\label{lem-comp6}
Assertion (iii)' implies assertion (i)' in Theorem \ref{thm-compact}. 
\end{lemma}

\begin{proof}
We show that if $f\in \mathcal{H}$ is orthogonal to all functions of the form $K\ast_Y g$ where $K$ ranges over $\Gamma$-invariant functions in $L^\infty(X\times_Y X)$ and $g$ ranges over $L^2(X)$, then $f=0$. 
Let $K$ be the unique element of minimal norm in the closed convex hull of $\orb(f\otimes \bar f)$ in  $L^2(X\times_Y X)$ which is $\Gamma$-invariant by Theorem \ref{ab-thm}. 
Set  $K_M=K1_{|K|\leq M}$ for $M\geq 0$. 
By hypothesis, $f$ is orthogonal to $K_M\ast_Y \bar f$. 
We can rewrite this, using the disintegration $(\mu_y)_{y\in\tilde{Y}}$ of $\tilde{\mu}$, as
\begin{align*}
0&=\int_{\tilde{X}}  f(x) \left(\int_{\tilde{X}} K_M(x,x') \bar f(x') d\mu_{\tilde{\pi}(x)}(x')\right) d\tilde{\mu}(x) \\ 
&= \int_{\tilde{Y}} \int_{\tilde{X}} f(x) \left(\int_{\tilde{X}} K_M(x,x')  \bar f(x') d\mu_y(x')\right) d\mu_y(x) d\tilde{\nu}(y) \\
&= \int_{\tilde{X}\times\tilde{X}} f\otimes  \bar f K_M d\tilde\mu\times_{\tilde{Y}} \tilde\mu. 
\end{align*}
Thus, $ f\otimes \bar f$ is orthogonal to $K_M$ in $L^2(X\times_Y X)$. 
As $K_M$ is $\Gamma$-invariant, $(T^\gamma\times T^\gamma)^* f\otimes \bar f$ is orthogonal to $K_M$ for all $\gamma\in \Gamma$ as well. 
Thus $K$ must be orthogonal to all $K_M$, which implies that $K_M=0$ for all $M\geq 0$, and therefore also $K=0$.  

Now choose a finite set $\mathcal{F}$ such that the property in (iii)' is satisfied for $f$ and an arbitrary $\eps>0$.  
Let $(K_n)$ be a sequence in the convex hull of $\orb(f\otimes \bar f)$ such that $\lim \|K_n\|_{L^2(X\times_Y X)}\to 0$.  
By Cauchy--Schwarz, we also have $\langle K_n, g\otimes \bar g\rangle_{L^2(X\times_Y X)}\to 0$ for all $g\in\mathcal{F}$.  
Expanding the inner product and choosing $n$ large enough, we can always obtain a $\gamma\in\Gamma$ such that 
\[
\sum_{g\in\mathcal{F}} \|\cip{(T^{\gamma})^*f}{g}\|_{L^2(Y)}
\]
is arbitrarily small. Thus for any $\delta>0$ we can find a set $E\in\Baire(\tilde{Y})$ with $\tilde\nu(E^c)<\delta$ such that 
\[
|\cip{(T^{\gamma})^*f}{g}| <\eps \quad \text{ on } E
\]
for all $g\in\mathcal{F}$ for some $\gamma$. 
By (iii)', we also have for all $\gamma\in\Gamma$
\[
\min_{g\in\mathcal{F}} \cnorm{(T^{\gamma})^*f -g} <\eps. 
\] 
By the triangle inequality, 
\[
\cnorm{f} 1_{T^{-\gamma}(E)}=\cnorm{(T^\gamma)^*f} 1_{E}< 3 \eps
\]
for some $\gamma$. 
Since $\eps,\delta$ were chosen arbitrarily it follows that $f=0$, and this finishes the proof of Theorem \ref{thm-compact}. 
\end{proof}

\section{The uncountable Mackey-Zimmer theorem and isometric extensions}\label{sec-isometric}

Under the hypothesis of ergodicity, relatively compact extensions are isomorphic to homogeneous skew-product extensions.   
This characterization is a relative version of the Halmos-von Neumann theorem (proved for $\Z$-actions by Halmos and von Neumann \cite{halmosvonneumann}, then generalized by Mackey \cite{mackey-hvn} to second-countable locally compact group actions, see \cite[Theorem 17.5]{EFHN} for an uncountable Halmos-von Neumann theorem for ergodic Markov semi-group actions). 
The relative version of the Halmos-von Neumann theorem was obtained by Zimmer \cite{zimmer1976ergodic} for actions of second-countable locally compact groups on standard Borel spaces. In this section, we extend Zimmer's theorem to systems of uncountable complexity. 
The key ingredient in the proof is the Mackey-Zimmer theorem classifying ergodic homogeneous extensions. We use the uncountable Mackey-Zimmer theorem \cite{jt20} established by Tao and the author. We point out that similar results to the ones in this section were obtained by Ellis \cite[\S 5]{ellis}, though with different methods. In short, the methods of \cite{ellis} rely on topological dynamics and the theory of Banach bundles of spaces of continuous functions which is akin to the methods used in \cite{ehk}, whereas our methods are ergodic-theoretic and based off measure theory and topos theory. 

We start by recalling the uncountable Mackey-Zimmer theorem \cite{jt20}. 
Let $K$ be a compact Hausdorff group and $(Y,\nu)$ a $\OpProbAlg$-space. 
We denote by $\Cond_Y(K)$ the set of all $\AbsMes$-morphisms from $Y$ to $\Baire(K)$. 
Let $\CondTilde_Y(K)$ denote the set of continuous functions from $\tilde{Y}$ to $K$. 
It is a remarkable property of the canonical model that we have the isomorphism \begin{equation}\label{eq-canrep}\Cond_Y(K)\equiv \CondTilde_Y(K),\end{equation} 
see \cite[Proposition 7.9]{jt-foundational} for a proof (here isomorphism is understood in the sense of sets, that is, there is a bijection between the sets $\Cond_Y(K)$ and $\CondTilde_Y(K)$). 
The set $\CondTilde_T(K)$ has the structure of a group defining the group law in a pointwise way. 
Therefore, also $\Cond_Y(K)$ has a group structure. 
We introduce $\OpProbAlgG$-cocycles, homogeneous skew-products and extensions. 

\begin{definition}\label{def-skewprod}
Let $\mathcal{Y}=(Y,\nu,S)$ be a $\OpProbAlgG$-system, $K$ a compact Hausdorff group, and $L\leq K$ a closed subgroup. 
\begin{itemize}
    \item[(i)] A $K$-valued \emph{$\OpProbAlgG$-cocycle} is a family $\rho=(\rho_\gamma)_{\gamma\in\Gamma}$ of elements $\rho_\gamma\in \Cond_Y(K)$ satisfying the $\OpProbAlgG$-cocycle property 
    $$\rho_{\gamma\gamma'}=(\rho_\gamma\circ S^{\gamma'})\rho_{\gamma'}$$
    for all $\gamma,\gamma'\in\Gamma$. 
    Any $K$-valued $\OpProbAlgG$-cocycle $\rho=(\rho_\gamma)_{\gamma\in\Gamma}$ has a canonical representation $\tilde{\rho}=(\tilde{\rho}_\gamma)_{\gamma\in\Gamma}$ where $\tilde{\rho}_\gamma\in \CondTilde_Y(K)$ is defined by the isomorphism \eqref{eq-canrep} such that  $\tilde{\rho}=(\tilde{\rho}_\gamma)_{\gamma\in\Gamma}$ is a $\CHProbG$-cocycle on $\tilde{Y}$. 
    \item[(ii)] Let $\rho=(\rho_\gamma)_{\gamma\in\Gamma}$ be a $\OpProbAlgG$-cocycle with canonical representation $\tilde{\rho}=(\tilde{\rho}_\gamma)_{\gamma\in\Gamma}$. We denote by $\tilde{Y}\rtimes_{\tilde{\rho}} K/L$ the \emph{$\CHProbG$-homogeneous skew-product} defined by the following data: 
    \begin{itemize}
        \item The homogeneous space $K/L$ is equipped with Baire $\sigma$-algebra and Haar measure. 
        \item We equip $\tilde{Y}\times K/L$ with the product Baire probability measure.  
        \item The $\CHProbG$-action $T:\Gamma\to \Aut(\tilde{Y}\times K/L)$ is defined by $$T^\gamma(y,kL)=(S^\gamma y,\tilde{\rho}_\gamma(y) kL)$$
        for all $(y,kL)\in \tilde{Y}\times K/L$ and $\gamma\in\Gamma$. 
    \end{itemize}
    We define the \emph{$\OpProbAlgG$-homogeneous skew-product} $\mathcal{Y}\rtimes_\rho K/L$ to be the $\OpProbAlgG$-system $$\mathtt{AlgAbs}(\tilde{Y}\rtimes_{\tilde{\rho}} K/L).$$  
    \item[(iii)]    A \emph{$\OpProbAlgG$-homogeneous extension} of $\mathcal{Y}$ by $K/L$ is a tuple $(\mathcal{X},\pi,\theta,\rho)$ where $\mathcal{X}=(X,\mu,T)$ is a $\OpProbAlgG$-system, $\pi:(X,\mu,T)\to (Y,\nu,S)$ is a $\OpProbAlgG$-extension, $\theta\in \Cond_X(K/L)$ is a vertical coordinate such that $\pi,\theta$ jointly generate the $\SigmaAlg$-algebra $X$, and $\rho=(\rho_\gamma)_{\gamma\in\Gamma}$ is a $\OpProbAlgG$-cocycle such that $$\theta\circ T^\gamma = (\rho_\gamma\circ \pi)\theta$$
    for all $\gamma\in \Gamma$ using the natural action of $\Cond_X(K)$ on $\Cond_X(K/L)$. 
\end{itemize}
\end{definition}

\begin{theorem}[Uncountable Mackey-Zimmer]\label{thm-mz} \cite{jt20} 
Let $\mathcal{Y}=(Y,\nu,S)$ be an ergodic $\OpProbAlgG$-system and $K$ be a compact Haudorff group.
 Every ergodic $\OpProbAlgG$-homogeneous extension $\mathcal{X}$ of $\mathcal{Y}$ by $K/L$ for some compact subgroup $L$ of $K$ is isomorphic in $\OpProbAlgG$ to a $\OpProbAlgG$-homogeneous skew-product $\mathcal{Y} \rtimes_\rho H/M$ for some compact subgroup $H$ of $K$, some compact subgroup $M$ of $H$, and some $H$-valued $\OpProbAlgG$-cocycle $\rho$.  
\end{theorem}

We apply the uncountable Mackey-Zimmer theorem to establish a geometric characterization of relatively compact extensions: 

\begin{theorem}\label{thm-isometric}
If $\pi:(X,\mu,T)\to (Y,\nu,S)$ is an ergodic relatively compact $\OpProbAlgG$-extension, then $(X,\mu,T)$ is $\OpProbAlgG$-isomorphic to a homogeneous skew-product extension $\mathcal{Y}\rtimes_\rho H/M$ for a compact Hausdorff group $H$, a closed subgroup $M$ of $H$, and a $H$-valued $\OpProbAlgG$-cocycle $\rho$.  
\end{theorem}

In the proof, we need some properties about cocycles with values in quotient and product spaces which are recorded in the following lemma. 

\begin{lemma}\label{lem-prodquot}
Let $A$ be an index set and for every $\alpha\in A$, let $K_\alpha$ be a compact Hausdorff group and $L_\alpha$ a closed subgroup of  $K_\alpha$. Let $(Y,\nu)$ be a $\OpProbAlg$-space.
Then the following identities hold:
    \begin{align*}
\Cond_Y(\prod_{\alpha\in A} K_\alpha)&=\prod_{\alpha\in A} \Cond_Y(K_\alpha)\\
\Cond_Y(\prod_{\alpha\in A} (K_\alpha/L_\alpha))&=\Cond_Y(\prod_{\alpha\in A} K_\alpha/\prod_{\alpha\in A} L_\alpha)
\end{align*}
\end{lemma}

\begin{proof}
The first identity is proved in \cite[Corollary 3.5]{jt19}. 
Routine verification shows that $$\prod_{\alpha\in A} (K_\alpha/L_\alpha)\equiv\prod_{\alpha\in A} K_\alpha / \prod_{\alpha\in A} L_\alpha$$ is a $\CH$-isomorphism. 
This proves the second identity.  
\end{proof}

\begin{proof}[Proof of Theorem \ref{thm-isometric}]
Let $\pi:(X,\mu,T)\to (Y,\nu,S)$ be a relatively compact $\OpProbAlgG$-extension of ergodic systems. 
Let $(\mathcal{M}_\alpha)_{\alpha\in A}$ be the family of all $\Gamma$-invariant finitely generated $L^0(Y)$-submodules of $\cltwo$. 
Fix $\alpha\in A$. 
Since $\mathcal{M}_\alpha$ is $\Gamma$-invariant, the map $\cdim(\mathcal{M}_\alpha)$ is also $\Gamma$-invariant and by ergodicity $\cdim(\mathcal{M}_\alpha)\equiv d_\alpha$ for some integer $d_\alpha\geq 1$.  
By Proposition \ref{prop-GSP}, we find $f^\alpha_1,\ldots,f^\alpha_{d_\alpha}\in \mathcal{M}_\alpha$ such that $\cnorm{f^\alpha_i}=1$ for all $i=1,\ldots,d_\alpha$ and $\cip{f^\alpha_i}{f^\alpha_j}=0$ whenever $i\neq j$. 
In fact, we can choose $f^\alpha_1,\ldots,f^\alpha_{d_\alpha}$ such that all $f^\alpha_i$ take values in the unit circle $\mathbb{S}$. 
For any $\gamma\in\Gamma$, then $(\tilde T^\gamma)^*f^\alpha_1,\ldots, (\tilde T^\gamma)^*f^\alpha_{d_\alpha}$ also satisfies $\cnorm{(\tilde T^\gamma)^*f^\alpha_i}=1$ for all $i=1,\ldots,d_\alpha$ and $\cip{(\tilde T^\gamma)^*f^\alpha_i}{(\tilde T^\gamma)^*f^\alpha_j}=0$ whenever $i\neq j$.   
Let 
$$\Lambda_\gamma^\alpha=
\begin{bmatrix}
a_{11}^\gamma & \ldots & a_{1d_{\alpha}}^\gamma \\
\vdots & \ddots & \vdots \\
a_{d_{\alpha}1}^\gamma & \ldots & a_{d_{\alpha}d_{\alpha}}^\gamma 
\end{bmatrix}
$$
where $a^\gamma_{ij}\in L^0(Y)$ is such that $f^\alpha_i=\sum_{j=1}^{d_\alpha} a_{ij}^\gamma (\tilde T^\gamma)^*f^\alpha_j$ for all $i=1,\ldots,d_\alpha$. 
We can identify $\Lambda_\gamma^\alpha$ with an element of the function space $L^0(Y\to \mathbb{U}(d_{\alpha}))$ of equivalence classes of $\mathbb{U}(d_{\alpha})$-valued random variables on $\tilde{Y}$, where $\mathbb{U}(d_{\alpha})$ is the group of $d_{\alpha}\times d_{\alpha}$ unitary matrices. 
We obtain the identity $F^{\alpha}=\Lambda_\gamma^\alpha F^{\alpha}_\gamma$ where $F^{\alpha}=(f^{\alpha}_1,\ldots,f^{\alpha}_{d_{\alpha}})$ and $F^{\alpha}_\gamma=((\tilde T^\gamma)^*f^{\alpha}_1,\ldots, (\tilde T^\gamma)^*f^{\alpha}_{d_{\alpha}})$. 
Let $\rho^\alpha_\gamma\coloneqq ((\Lambda_\gamma^\alpha)^*)^\op$ which is an element of $\Cond_Y(\mathbb{U}(d_{\alpha}))$. 
We thus obtain a $\mathbb{U}(d_{\alpha})$-valued $\OpProbAlgG$-cocycle $\rho^\alpha=(\rho^\alpha_\gamma)_{\gamma\in\Gamma}$ on $(Y,\nu,S)$.  
Define the $\AbsMes$-morphism $\theta_\alpha:Y\to \Baire(\mathbb{S}^{2d_{\alpha}-1})$ by $\theta_\alpha\coloneqq ((\frac{1}{\sqrt{d_{\alpha}}} F^\alpha)^*)^\op$. 

Now define $\rho=(\rho_\gamma)_{\gamma\in\Gamma}$ by $\rho_\gamma=(\rho_\gamma^\alpha)_{\alpha\in A}$, $\theta=(\theta_\alpha)_{\alpha\in A}$, $K=\prod_{\alpha\in A} \mathbb{U}(d_\alpha)$ and $L=\prod_{\alpha\in A} \mathbb{U}(d_\alpha-1)$. 
By  Lemma \ref{lem-prodquot}, $\theta\in \Cond_Y(K/L)$ is a vertical coordinate such that $\pi,\theta$ jointly generate the $\SigmaAlg$-algebra $X$, and $\rho=(\rho_\gamma)_{\gamma\in\Gamma}$ is a $\OpProbAlgG$-cocycle such that $\theta\circ T^\gamma = (\rho_\gamma\circ \pi)\theta$. 
The claim follows from Theorem \ref{thm-mz}. 
\end{proof}

\begin{remark}\label{rem-compgrext}
Conversely, every (not necessarily ergodic) $\OpProbAlgG$-homogeneous skew-product $\mathcal{Y}\rtimes_\rho K/L$ is a relatively compact $\OpProbAlgG$-extension of $\mathcal{Y}$. 
It is sufficient to show this for $\OpProbAlgG$-group skew-products $\mathcal{Y}\rtimes_\rho K$. Indeed, if $\mathcal{Y}\rtimes_\rho K\to \mathcal{Y}$ is a relatively compact $\OpProbAlgG$-extension, then so is $\mathcal{Y}\rtimes_\rho K/L\to \mathcal{Y}$ since we have a $\OpProbAlgG$-extension $\mathcal{Y}\rtimes_\rho K\to \mathcal{Y}\rtimes_\rho K/L$. 

Now $L^2(Y\rtimes_\rho K)$ is the Hilbert space tensor product of $L^2(Y)$ and $L^2(K)$. 
By the Peter-Weyl theorem, $L^2(K)$ is the closure of the union of finite-dimensional $K$-invariant subspaces. 
For any $g$ in a finite-dimensional $K$-invariant subspace of $L^2(K)$ and $f\in L^\infty(Y)$, the tensor $f\otimes g$ lies in a $\Gamma$-invariant finitely generated $L^\infty(Y)$-submodule of $L^2(Y\rtimes_\rho K)$. 
The claim follows from Theorem \ref{thm-compact}. 
\end{remark}

\begin{remark}
We call $(\mathcal{Y}_\alpha)_{\alpha\leq \beta}$ in Theorem \ref{thm-FZ} the \emph{Furstenberg tower} associated to the $\OpProbAlgG$-system $\mathcal{X}$ and call $\beta$ the \emph{length} of this tower.  
If $\mathcal{X}$ is a $\OpProbAlgG$-system of countable complexity, then the length of the Furstenberg tower of $\mathcal{X}$ must be a countable ordinal since $L^2(X)$ is separable. 
Beleznay and Foreman \cite{foreman} proved that any countable ordinal can be realized as the length of the Furstenberg tower of a $\OpProbAlgG$-system of countable complexity.
By a transfinite recursion, we can build out of $\OpProbAlgG$-group skew-products $\OpProbAlgG$-systems of Furstenberg tower of arbitrary length (cf. Remark \ref{rem-compgrext}). We leave the details to the interested reader. 
\end{remark}

\begin{remark}
Austin  establishes a relatively ergodic version of Theorem \ref{thm-isometric} for systems of countable complexity in \cite[\S 4]{austin-fund}. We think that a significantly heavier application of the topos-theoretic machinery will prove a generalization of Austin's result to systems of arbitrary (not necessarily countable) complexity. We hope to work out the details of this in future work. \end{remark}

\section{Dichotomy}\label{sec-structure}

Given any $\OpProbAlgG$-system $X$, it is well known that $L^2(X)$ is the orthogonal sum of the compact and weakly mixing factors of $X$, e.g., see Chapter 1.7 of these lecture notes \cite{peterson11} by Peterson. 
This is an ergodic theoretic manifestation of the dichotomy between structure and randomness. 
In this section, we establish a relative or conditional version of this dichotomy for $\OpProbAlgG$-extensions. 
Let us first define relatively weakly mixing functions and extensions. 
\begin{definition}
Let $\mathcal{X}=(X,\mu,T)$ and $\mathcal{Y}=(Y,\nu,S)$ be $\OpProbAlgG$-systems. 
Let $\pi:\mathcal{X}\to \mathcal{Y}$ be a $\OpProbAlgG$-extension. 
A function $f\in L^2(X)$ with $\E(f|Y)=0$ is said to be \emph{relatively weakly mixing} if for all $\eps>0$ there exists $\gamma\in \Gamma$ such that 
\[
\|\cip{(T^\gamma)^*f}{f}\|_{L^2(Y)}<\eps. 
\]
Moreover, we say that $\mathcal{X}$ is a \emph{relatively weakly mixing $\OpProbAlgG$-extension} of $\mathcal{Y}$ if all $f\in L^2(X)$ with $\E(f|Y)=0$ are relatively weakly mixing. 
Finally, we denote by $\WM$ the subspace of all relatively mixing functions in $L^2(X)$.  
\end{definition}
 
We denote by $\AP$ the $L^2$ closure of all $\Gamma$-invariant finitely generated $L^\infty(Y)$-submodules of $L^2(X)$. 
The relative or conditional version of the classical dichotomy between the compact and weakly mixing parts of a system is  
\begin{equation}\label{eq:dichotomy}
L^2(X)=\AP\oplus \WM
\end{equation}
where the sum is in the sense of the Hilbert space $L^2(X)$. 
This relative dichotomy is well understood for systems of countable measure-theoretic complexity, e.g., see \cite[\S 2.14]{tao2009poincare} for an exposition in the case of $\Z$-actions.    
Theorem \ref{thm-dichotomy1} extends this relative dichotomy to systems of uncountable complexity which we restate next for the convenience of the reader.  

\begin{theorem}[Uncountable relative dichotomy]\label{thm-dichotomy}
Let $\mathcal{X}=(X,\mu,T)$ and $\mathcal{Y}=(Y,\nu,S)$ be $\OpProbAlgG$-systems and $\pi: \mathcal{X}\to \mathcal{Y}$ be a $\OpProbAlgG$-extension.  
Exactly one of the following statements is true.  
\begin{itemize}
\item[(i)] $\pi: \mathcal{X}\to \mathcal{Y}$ is a relatively weakly mixing $\OpProbAlgG$-extension.  
\item[(ii)] There exist a $\OpProbAlgG$-system $\mathcal{Z}=(Z,\lambda,R)$ and $\OpProbAlgG$-extensions $\phi:\mathcal{X}\to \mathcal{Z}$ and $\psi:\mathcal{Z}\to \mathcal{Y}$ such that $\psi$ is a non-trivial relatively compact $\OpProbAlgG$-extension.  
\end{itemize}
\end{theorem}

\begin{proof}
If (i) is false, there are $f\in L^2(X)$ with $\E(f|Y)=0$ and $\eps>0$ such that for all $\gamma\in \Gamma$
\begin{equation}\label{1eq2}
 \langle (T^\gamma)^*f\otimes \bar f, f\otimes \bar f\rangle_{L^2(X\times_Y X)}=\|\cip{(T^\gamma)^*f}{f}\|^2_{L^2(Y)}\geq \eps. 
\end{equation}
Let $K$ be the unique element of minimal norm in the closed convex hull of $\orb(f\otimes \bar f)$. 
By Theorem \ref{ab-thm} $K$ is $\Gamma$-invariant, and by \eqref{1eq2} $K$ is non-trivial. 
Since $\E(f|Y)=0$, there must exist $g\in L^2(X)$ such that $K\ast_Y g\in L^2(X)\backslash L^2(Y)$. 
Let $\mathcal{A}=\{K\ast_Y g\in L^\infty(X): g\in L^2(X)\}$. 
By Theorem \ref{thm-compact}, every $K\ast_Y g\in\mathcal{A}$ is contained in a closed $\Gamma$-invariant finitely generated $L^\infty(Y)$-submodule of $L^2(X)$. 
It is not difficult to check that $\mathcal{A}$ is a $\Gamma$-invariant vector subspace of $L^\infty(X)$ closed under complex conjugation and multiplication. 
Thus $\mathcal{A}$ has the structure of a $\vonNeumannG$-system such that $\Phi:\mathcal{A}\to \Linfty(X,\mu,T)$ and $\Psi:\Linfty(Y,\nu,S)\to \mathcal{A}$ are $\vonNeumannG$-factors with $\Psi$ non-trivial. 
Then $(Z,\lambda, R)\coloneqq \Idem(\mathcal{A})$, $\phi\coloneqq \Idem(\Phi)$ and $\psi\coloneqq \Idem(\Psi)$ satisfy (ii). 

Now let $f$ be relatively weakly mixing and $g\in \AP$. 
Without loss of generality, we may assume that $f$ is bounded and $g$ is an element of an invariant finitely generated $L^\infty(Y)$-submodule $\mathcal{M}$ of $L^2(X)$. 
For any $\gamma\in\Gamma$ and $h\in \mathcal{M}$, we have 
\begin{align*}
    \|\cip{f}{g}\|_{L^2(Y)}&=\|\cip{(T^\gamma)^*f}{(T^\gamma)^*g}\|_{L^2(Y)}\\
    &\leq \|\cip{(T^\gamma)^*f}{(T^\gamma)^*g - h}\|_{L^2(Y)}+\|\cip{(T^\gamma)^*f}{h}\|_{L^2(Y)} \\
    &\leq \|\cnorm{(T^\gamma)^*f}\cnorm{(T^\gamma)^*g - h}\|_{L^2(Y)} +\|\cip{(T^\gamma)^*f}{h}\|_{L^2(Y)} \\
    &\leq \|f\|_{L^\infty(X)}\|\cnorm{(T^\gamma)^*g - h}\|_{L^2(Y)} + \|\cip{(T^\gamma)^*f}{h}\|_{L^2(Y)}
\end{align*}
    where the second inequality follows from the conditional Cauchy--Schwarz inequality.  
    For any $\delta>0$, we can first choose $h$ and then $\gamma$ such that the last term is $<\delta$ uniformly in $\gamma$. 
    Hence $\|\cip{f}{g}\|_{L^2(Y)}=0$, and thus $\langle f,g\rangle_{L^2(X)}=0$. 
    It remains to show that if $f\not\in \WM$, then $f$ must correlate with some $g\in \AP$, but this follows by a similar argument as in the first part of the proof. 
\end{proof}

\begin{remark}\label{rem:wm}
A $\OpProbAlgG$-extension $\pi:\mathcal{X}\to \mathcal{Y}$ is said to be relatively ergodic if $\Inv(\mathcal{X})=\Inv(\mathcal{Y})$ (after identifying $\Inv(\mathcal{Y})$ with a subalgebra of $\mathcal{X}$). 
Using the relative dichotomy \eqref{eq:dichotomy}, it is not difficult to show that a $\OpProbAlgG$-extension $\pi:\mathcal{X}\to \mathcal{Y}$ is relatively weakly mixing if and only if $\mathcal{X}\times_\mathcal{Y} \mathcal{X}$ is a relatively ergodic extension of $\mathcal{Y}$.  
\end{remark}

\begin{remark}
The relative dichotomy \eqref{eq:dichotomy} can be extended to a conditional Halmos-von Neumann type decomposition of the conditional Hilbert space $\cltwo$ which is the external interpretation of the internal Halmos-von Neumann theorem of the topos $\mathtt{Sh}(Y)$.  
In fact, this provides an alternative route to the proof of Theorem \ref{thm-dichotomy}. 
To be more precise, construct from a $\OpProbAlgG$-extension $\pi:\mathcal{X}\to \mathcal{Y}$ an internal measure-preserving dynamical system in $\mathtt{Sh}(Y)$ (see Remark \ref{rem-internal}), 
apply to that internal system the internal Halmos-von Neumann theorem, and then externally interpret the resulting internal orthogonal decomposition. This reasoning which basically reduces to a translation process between the external and internal universes is enabled by a logical Boolean transfer principle powerful enough to yield a Halmos-von Neumann theorem in the internal logic of the topos $\mathtt{Sh}(Y)$.
In fact, the above proof of Theorem \ref{thm-dichotomy} is a reflection and validation of this transfer principle. 
\end{remark}

We can use Theorem \ref{thm-dichotomy} to prove our main result, Theorem \ref{thm-FZ}, which we now restate:

\begin{theorem}
Let $\mathcal{X}=(X,\mu,T)$ be a $\OpProbAlgG$-system. 
Then there is an ordinal number $\beta$ such that for each ordinal number $\alpha\leq \beta$ there are a $\OpProbAlgG$-system $\mathcal{Y}_\alpha=(Y_\alpha,\nu_\alpha,S_\alpha)$ and a $\OpProbAlgG$-factor $\pi_\alpha\colon \mathcal{X}\to \mathcal{Y}_\alpha$, and for every successor ordinal number $\alpha+1\leq \beta$ there is $\OpProbAlgG$-factor $\pi_{\alpha+1,\alpha} \colon \mathcal{Y}_{\alpha+1}\to \mathcal{Y}_{\alpha}$ with the following properties:
\begin{itemize}
\item[(i)] $Y_0$ is the trivial algebra.  
\item[(ii)] $\pi_{\alpha+1,\alpha} \colon \mathcal{Y}_{\alpha+1}\to \mathcal{Y}_{\alpha}$ is a non-trivial relatively compact $\OpProbAlgG$-extension. 
\item[(iii)] $\mathcal{Y}_\alpha$ is the inverse limit of the systems $\mathcal{Y}_{\eta}$, $\eta<\alpha$ for every limit ordinal number $\alpha\leq \beta$, in the sense that $Y_\alpha$ is generated by $\bigcup_{\eta<\alpha} Y_\eta$ as a $\sigma$-complete Boolean algebra.     
\item[(iv)] $\pi_{\beta} \colon \mathcal{X}\to \mathcal{Y}_\beta$ is a relatively weakly mixing $\OpProbAlgG$-extension.  
\end{itemize}
\end{theorem}

\begin{proof}
Let $\mathcal{Y}_0$ be the trivial $\OpProbAlgG$-system, and let $\pi_0\colon \mathcal{X}\to \mathcal{Y}_0$ be the corresponding $\OpProbAlgG$-factor map. 
By Theorem \ref{thm-dichotomy}, $\pi_0\colon \mathcal{X}\to \mathcal{Y}_0$ is either a relatively weakly mixing $\OpProbAlgG$-extension in which case we set $\beta=0$, or there are a $\OpProbAlgG$-system $\mathcal{Z}$ and $\OpProbAlgG$-extensions $\phi\colon \mathcal{X}\to \mathcal{Z}$ and $\psi\colon \mathcal{Z}\to \mathcal{Y}_0$ such that $\psi \colon \mathcal{Z} \to \mathcal{Y}_0$ is a non-trivial relatively compact $\OpProbAlgG$-extension of $\mathcal{Y}_0$. 
Then set $\mathcal{Y}_1\coloneqq\mathtt{AP_{X|Y_0}}$, let $\pi_1\colon \mathcal{X}\to \mathcal{Y}_1$ and $\pi_{1,0}\colon \mathcal{Y}_1\to \mathcal{Y}_0$ be the corresponding $\OpProbAlgG$-factors, and repeat the previous step with $\mathcal{X}$ and $\mathcal{Y}_1$ instead of $\mathcal{Y}_0$ using Theorem \ref{thm-dichotomy}. This process may end here if $\pi_1\colon \mathcal{X}\to \mathcal{Y}_1$ is a relatively weakly mixing $\OpProbAlgG$-extension and we set $\beta=1$, or we may need to continue and repeat the previous step again.  
Now repeating the previous steps in a transfinite recursion, while passing to inverse limits at limit ordinals, this process will eventually terminate at an ordinal number $\beta$.  
\end{proof}

\end{document}